\documentclass[11pt,english]{article}

\usepackage[margin=2.29cm]{geometry}

\usepackage{amsthm}
\usepackage{amsmath}
\usepackage{amssymb}
\usepackage{setspace}
\usepackage{mathtools}
\usepackage{graphicx}
\usepackage[hidelinks]{hyperref}
\usepackage{thm-restate}
\usepackage{cleveref}
\usepackage{graphicx}
\usepackage{enumitem}
\setlist[itemize]{leftmargin=*} 
\usepackage{framed}
\usepackage{subcaption}

\usepackage{floatrow}

\usepackage[square,sort,comma,numbers]{natbib}
\theoremstyle{plain}

\newtheorem*{thm*}{Theorem}
\newtheorem{thm}{Theorem}
\Crefname{thm}{Theorem}{Theorems}
\numberwithin{thm}{section}

\newtheorem*{lem*}{Lemma}
\newtheorem{lem}[thm]{Lemma}
\Crefname{lem}{Lemma}{Lemmas}

\newtheorem*{claim*}{Claim}
\newtheorem{claim}[thm]{Claim}
\crefname{claim}{Claim}{Claims}
\Crefname{claim}{Claim}{Claims}

\newtheorem{prop}[thm]{Proposition}
\Crefname{prop}{Proposition}{Propositions}

\newtheorem{cor}[thm]{Corollary}
\crefname{cor}{Corollary}{Corollaries}

\crefname{conj}{Conjecture}{Conjectures}

\Crefname{qn}{Question}{Questions}

\Crefname{obs}{Observation}{Observations}

\Crefname{ex}{Example}{Examples}

\theoremstyle{definition}

\Crefname{prob}{Problem}{Problems}

\newtheorem{defn}[thm]{Definition}
\Crefname{defn}{Definition}{Definitions}

\theoremstyle{remark}

\renewenvironment{proof}[1][]{\begin{trivlist}
\item[\hspace{\labelsep}{\bf\noindent Proof#1.\/}] }{\qed\end{trivlist}}

\newcommand{\remove}[1]{}

\newcommand{\R}{\overleftrightarrow{r}} 
\newcommand{\RT}{\overrightarrow{r}} 
\newcommand{\eps}{\varepsilon}

\newcommand{\dirpath}[1]{\overrightarrow{P_{#1}}}
\newcommand{\dircomp}[1]{\overleftrightarrow{K_{#1}}}
\DeclareMathOperator{\leaf}{lf}

\expandafter\def\expandafter\normalsize\expandafter{%
    \normalsize
    \setlength\abovedisplayskip{4pt}
    \setlength\belowdisplayskip{4pt}
    \setlength\abovedisplayshortskip{4pt}
    \setlength\belowdisplayshortskip{4pt}
}

\begin{document}


\vspace{-0.9cm}
\title{Directed Ramsey number for trees}
\date{\vspace{-5ex}}
\author{
	Matija Buci\'c\thanks{
	    Department of Mathematics, 
	    ETH, 
	    8092 Zurich;
	    e-mail: \texttt{matija.bucic}@\texttt{math.ethz.ch}.
	}
	\and
    Shoham Letzter\thanks{
        ETH Institute for Theoretical Studies,
        ETH,
        8092 Zurich;
        e-mail: \texttt{shoham.letzter}@\texttt{eth-its.ethz.ch}.
    }
    \and
	Benny Sudakov\thanks{
	    Department of Mathematics, 
	    ETH, 
	    8092 Zurich;
	    e-mail: \texttt{benjamin.sudakov}@\texttt{math.ethz.ch}.     
	}
}

\maketitle
\begin{abstract}
    \setlength{\parskip}{\medskipamount}
    \setlength{\parindent}{0pt}
    \noindent In this paper, we study Ramsey-type problems for directed graphs. 
    We first consider the $k$-colour \textit{oriented Ramsey number} of $H$, denoted by $\RT(H,k)$, which is the least $n$ for which every $k$-edge-coloured tournament on $n$ vertices contains a monochromatic copy of $H$. We prove that $ \RT(T,k) \le c_k|T|^k$ for any oriented tree $T$. This is a generalisation of a similar result for directed paths by Chv\'atal and by Gy\'arf\'as and Lehel, and answers a question of Yuster. In general, it is tight up to a constant factor.
    
    We also consider the $k$-colour \emph{directed Ramsey number} $\R(H,k)$ of $H$, which is defined as above, but, instead of colouring tournaments, we colour the complete directed graph of order $n$. Here we show that $ \R(T,k) \le c_k|T|^{k-1}$ for any oriented tree $T$, which is again tight up to a constant factor, and it generalises a result by Williamson and by Gy\'arf\'as and Lehel who determined the $2$-colour directed Ramsey number of directed paths.
\end{abstract}

\section{Introduction}

An \textit{oriented graph} is a directed graph in which between any two vertices there is at most one edge. The \emph{underlying graph} of a directed graph is the graph obtained by removing orientation from its edges.

One of the most classical results in the theory of directed graphs is the Gallai-Hasse-Roy-Vitaver theorem \cite{gallai1968directed, hasse1965algebraischen, roy1967nombre,  vitaver1962determination}, abbreviated here as the GHRV theorem, which states that any directed graph, whose underlying graph has chromatic number\footnote{Throughout the paper by chromatic number we always mean \textit{vertex} chromatic number.} at least $n$, contains a directed path of length $n-1$. Note that, by the length of a path we mean the number of \emph{edges} in the path.

It is natural to ask if there are directed graphs, other than the directed path, which are guaranteed to exist in any $n$-chromatic oriented graph. We note that if $H$ is a directed graph whose underlying graph contains a cycle, then for every $k$, the graph $H$ is not guaranteed to be a subgraph of every $k$-chromatic directed graph, due to the existence of graphs with arbitrarily large girth and chromatic number. Furthermore, if  $H$ contains a bidirected edge, it is clearly not a subgraph of any oriented graph. Hence, we remain with oriented trees (and forests). This question was first asked by Burr \cite{burr1980subtrees} in $1980$, who conjectured that every $(2n-2)$-chromatic digraph contains every oriented tree of order $n$. If true, this is best possible, as a regular tournament on $2n-3$ vertices is clearly $(2n-3)$-chromatic, but has maximum out-degree $n-2$, so it does not contain an out-directed star on $n$ vertices. The conjecture is still widely open and even a weaker version of it, where $2n-2$ is replaced by $c n$ for a large constant $c$, is not known. Burr showed that the statement holds for $(n-1)^2$-chromatic digraphs, and the best general result in this direction is due to Addario-Berry, Havet, Sales, Reed and Thomass\'e \cite{addario2013oriented} who proved it for $(n^2/2-n/2+1)$-chromatic digraphs. It is of note that the conjecture is open even for relatively simple trees, such as arbitrarily oriented paths (though the very special case of directed paths with two blocks is solved \cite{el-sahili2007two-blocks,addario-berry07two-blocks}).

The GHRV theorem has the following interesting application to Ramsey theory:
any $2$-edge-colouring of a tournament on $n^2+1$ vertices contains a monochromatic directed path of length $n$. Indeed, denote the subgraph of red edges by $T_R$ and the subgraph of blue edges by $T_B$. The union $T_R \cup T_B$ has chromatic number $n^2 + 1$. Recall that the chromatic number of a union of two graphs $F$ and $H$ is at most the product of the chromatic numbers of $F$ and $H$. Thus, without loss of generality, the chromatic number of $T_R$ is at least $n+1$, so by the GHRV theorem, it contains a directed path of length $n$. In particular, there is a monochromatic directed path of length $n$.

The above argument extends easily to more colours, showing that every $k$-edge-colouring of any tournament on $n^k + 1$ vertices contains a monochromatic path of length $n$. This was observed by Chv\'atal \cite{chvatal1972monochromatic} and Gy\'arf\'as and Lehel \cite{gyarfas1973ramsey} who also obtained a similar result for paths of different lengths. Gy\'arf\'as and Lehel also observed that there is a simple grid construction that shows that the bound $n^k + 1$ is tight.


Generalising this example, we define the $k$-colour \textit{oriented Ramsey number} of an oriented graph $H$, denoted by $\RT(H,k)$, to be the least integer $n$, for which every $k$-edge-coloured tournament on $n$ vertices contains a monochromatic copy of $H$. The aforementioned results show that $\RT(\dirpath{n}, k) = (n-1)^k+1$, where $\dirpath{n}$ is the directed path of order $n$.  

It is natural to  consider the extension of the above example to oriented trees. Namely, what can one say about the $k$-colour oriented Ramsey number of trees?
Even for $k = 1$, this questions is interesting and difficult. The celebrated conjecture of Sumner from $1971$, states that any tournament on $2n-2$ vertices contains any oriented tree on $n$ vertices (where $n \ge 2$; note that this is a special case of Burr's conjecture). It is clear that we cannot hope for a better result in general, because, as before, a regular tournament on $2n-3$ vertices does not contains an out-directed star of order $n$. Thomason \cite{thomason1989} proved that for sufficiently large $n$, every tournament on $n+1$ vertices contains every orientation of a path of order $n$, thus proving Sumner's conjecture for oriented paths (and large $n$).
H\"aggkvist and Thomason \cite{haggkvist1991} were the first to show that the statement for general trees holds for tournaments on at least $cn$ vertices, where $c$ is a constant. Following improvements by Havet \cite{havet2002trees} and Havet and Thomass\'e \cite{havet-thomasse2000median}, El Sahili \cite{el2004trees} used the notion of median orders, first used as a tool for Sumner's conjecture in \cite{havet-thomasse2000median}, to show that the statement holds for tournaments on $3n-3$ vertices; this is currently the best known upper bound for general $n$.
More recently, Sumner's conjecture was proved for sufficiently large $n$ by K\"uhn, Mycroft and Osthus \cite{kuhn2010proof}.

We note that Burr's conjecture, if true, would imply that $\RT(T, k) \le c_k|T|^k$ for every oriented tree, where $c_k$ is a constant that depends only on $k$. Indeed, consider a $k$-edge-colouring of a tournament on $(2n-3)^k+1$ vertices. Let $G_i$ be the subgraph consisting of all edges in colour $i$, for $i \in [k]$. Then there exists $i \in [k]$ for which $G_i$ has chromatic number at least $2n-2$; hence, by Burr's conjecture, $G_i$ contains a copy of $T$. In this paper we prove that $\RT(T, k) \le c_k|T|^k$ using a different approach.

\begin{thm}\label{thm:tournament-trees}
        There is a constant $c_k$ such that for any oriented tree $T$ the following holds.
        $$\RT(T,k) \le c_k|T|^k.$$
\end{thm}
    
This result is tight up to a constant factor for some trees, including directed paths. 
With \Cref{thm:tournament-trees} we make progress towards answering a question of Yuster \cite{Yuster20171}, who asked the following question: given $k$ and $t$, what is the minimum $n$ such that $\RT(T,k) \le n$ for every oriented tree $T$ of order at most $t$. Indeed, we show that this minimum is at most $c_k t^k$, which is a constant factor away from the lower bound $(t-1)^k$ which follows from  \Cref{thm:tournament-path} below. We note that Yuster proved that this lower bound is tight when $k \ge t \log t$. Note that if we fix the number of colours then Yuster's result only applies for trees of relatively small order, while our result applies for trees of any order. 
In the case of arbitrarily oriented paths, our proof of Theorem \ref{thm:tournament-trees} simplifies significantly and we obtain a stronger bound. 
Given an oriented path $P$, we denote by $l(P)$ the length of the longest directed subpath of $P$. Our second theorem provides a tight upper bound (up to a constant factor) on the oriented Ramsey number of an arbitrarily oriented path $P$, in terms of $l(P)$ and its length $n=|P|-1$.   

\begin{restatable}{thm}{thmpath}\label{thm:tournament-path}
    Let $P$ be a path of length $n$ with $l = l(P)$. Then the following holds for $c_k = 8^{k}\cdot k!$.
    $$n \cdot l^{k-1}  \le \RT(P,k) \le c_k \cdot n \cdot l^{k-1}.$$
\end{restatable}

It is worth noting that this result is generally stronger than the bound which would be implied by Burr's conjecture. 

An important distinction between the usual notion of Ramsey theory, and the variant of oriented Ramsey that we have introduced, is the fact that there is only one complete graph on $n$ vertices, while there are many tournaments on $n$ vertices. In particular, the answer to how large a monochromatic tree we can find in an edge colouring of a tournament $T$ on $n$ vertices, depends on $T$ as well as the colouring. For example, if $T$ is the transitive tournament on $n$ vertices, there is a $2$-edge-colouring with no monochromatic directed path of length $\sqrt{n}$. Contrasting this, in a recent paper \cite{paths-in-random-tournaments}, we prove that if $T$ is a random tournament on $n$ vertices then, with high probability, in every $2$-edge-colouring of $T$ there is a monochromatic path of length at least $\frac{cn}{\sqrt{\log n}}$, where $c > 0$ is an absolute constant.

Keeping this in mind, an underlying structure more analogous to undirected Ramsey case is the complete directed graph on $n$ vertices, which we denote by $\dircomp{n}$. Following Bermond \cite{bermond1974some}, we define the $k$-colour \textit{directed Ramsey number} of an oriented graph $H$, denoted by $\R(H,k)$, to be the least integer $n$ for which every $k$-edge-colouring of $\dircomp{n}$ contains a monochromatic copy of $H$; we emphasize that the edges $xy$ and $yx$ are allowed to have different colours. It is easy to see that the $k$-colour directed Ramsey number of a directed graph $G$, for $k\ge 2$, is finite if and only if $G$ is acyclic.  

Very few directed Ramsey numbers are known; here we outline some of the few results in this area. Harary and Hell \cite{harary74} introduced the notion of directed Ramsey numbers (for two colours) and determined its value for certain small graphs. Gy\'arf\'as and Lehel \cite{gyarfas1973ramsey} and independently Williamson \cite{williamson73}, deduce from a result of Raynaud \cite{raynaud1973circuit} that $\R(\dirpath{n}, 2) = 2n-3$ for $n \ge 3$. Bermond studied a related question for more colours; specifically, he considered the directed Ramsey number of a Hamiltonian graph (in one colour) vs.\ directed paths (of distinct lengths; in the remaining colours). He obtained a sharp bound for this Ramsey number, but his methods are not applicable, say, to the Ramsey number of directed paths. 





In this paper we determine the directed Ramsey number of trees, up to a constant factor.
\begin{thm}\label{thm:complete-trees}
    For every $k \ge 2$ there is a constant $c_k$, such that the following holds for every oriented tree $T$.     
    $$\R(T, k) \le c_k|T|^{k-1}.$$
\end{thm}

It is very interesting to note the different behaviour in comparison with the oriented case, where the Ramsey number of $T$ is $c_k|T|^k$. This is best illustrated by noticing that for two colours, the directed Ramsey number of a tree is linear in its order, while in the oriented setting it is quadratic. This difference prevents the usage of GHRV Theorem or of Burr's conjecture, even if it were true. Perhaps for this reason, our proof of \Cref{thm:complete-trees} is longer and more difficult than the proof of \Cref{thm:tournament-trees}, and requires additional ideas.

\subsection{Organisation of the paper}
In the following section, Section \ref{sec-col-tour}, we present our results for oriented Ramsey numbers; in particular, we prove Theorems \ref{thm:tournament-trees} and \ref{thm:tournament-path}. We turn our focus to directed Ramsey in Section \ref{sec-compl}, where we prove Theorem \ref{thm:complete-trees}; to do so, we built up on ideas that appear in Section \ref{sec-col-tour} but we also require new ingredients. We conclude the paper in Section \ref{sec-conc-rem} with some remarks and open problems.  

Throughout this paper, by a \emph{colouring} of a graph we mean an edge-colouring. Whenever we have a $k$-colouring, we denote the colours by $[k]=\{1,2,\ldots,k\}$, and when $k=2$ we call the colours red and blue. If a directed tree contains a bidirected edge, then both its oriented and directed Ramsey numbers, for $k \ge 2$, are infinite. Hence, throughout the paper, all the paths and trees are assumed to be oriented, i.e.\ they contain no bidirected edges. For a similar reason, we assume that  there are no loops. A \textit{directed path} is a path in which all edges follow the same direction, while by saying \textit{oriented path}, we stress that the edges are allowed to be oriented arbitrarily.

\section{Colouring Tournaments}\label{sec-col-tour}
In this section we focus on oriented Ramsey numbers; in particular, we prove Theorems \ref{thm:tournament-trees} and \ref{thm:tournament-path}. 

\subsection{Preliminaries}
We start by recalling the GHRV theorem, mentioned in the introduction.
\begin{thm}[Gallai \cite{gallai1968directed}, Hasse \cite{hasse1965algebraischen}, Roy \cite{roy1967nombre}, Vitaver \cite{vitaver1962determination}] \label{thm:ghrv}
	Every directed graph whose underlying graph has chromatic number at least $n$ contains a directed path of length $n-1$.
\end{thm} 

We also recall El-Sahili's result regarding Sumner's conjecture.
\begin{thm}[El Sahili \cite{el2004trees}] \label{thm:el-sahili}
	Every tournament on $3n-3$ vertices contains every oriented tree of order $n$.
\end{thm}

In order to enable inductive arguments, we need to work with asymmetric Ramsey numbers. To this end we define $\RT(G_1,G_2, \ldots, G_k)$ to be the least $n$ such that in any tournament of order $n$ whose edges are $k$-coloured, there is an $i$ such that we can find a copy of $G_i$ in the $i$-th colour.

Let $\leaf(T)$ denote the number of leaves of a tree $T$. We shall need the following definition.
\begin{defn} \label{def:core}
	Given a rooted oriented\footnote{In fact, this definition will be independent of the orientation of the edges.} tree $T$ we define the $k$\textit{-core} of $T$ to be the subtree $T'$ of $T$ consisting of all vertices with more than $\frac{1}{k}|T|$ descendants. Note that $T'$ can have at most $k$ leaves, as each non-root leaf has more than $\frac{1}{k}|T|$ descendants and these sets of descendants are mutually disjoint, so there are at most $k-1$ such leaves. Including the possibility of the root being a leaf, there are at most $k$ leaves in total. 
\end{defn}

\begin{figure}[h]
        \caption{A $5$-core of a tree on $19$ vertices.}
        \includegraphics{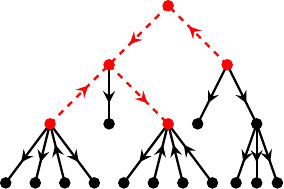}
        \label{fig:core}
\end{figure}

The following definition and a simple lemma, first proved in \cite{burr82}, will be very useful for the study of both the oriented and the directed Ramsey numbers.
\begin{defn}
	Given a graph $G$ and two disjoint subsets of vertices $X$ and $Y$, we say that the pair $(X,Y)$ is a $k$-\textit{mindegree pair of $G$} if every vertex in $X$ has at least $k$ out-neighbours in $Y$ and every vertex in $Y$ has at least $k$ in-neighbours in $X$. 
\end{defn}

\begin{lem} \label{lem:split}
    Let $G$ be a directed graph on $n$ vertices with at least $\eps \binom{n}{2}$ edges. Then $G$ contains an $\frac{\eps n}{8}$-mindegree pair.
\end{lem}

\begin{proof}
    Given two sets $A$ and $B$, we denote the number of edges directed from $A$ to $B$ by $e(A,B)$.
    Let $\{A, B\}$ be a partition of the vertices, obtained by putting each vertex in $A$ independently with probability $1/2$ (and putting it in $B$ otherwise). It is easy to see that the expected number of edges from $A$ to $B$ is $e(G)/4$. Hence, there is a choice of a partition $\{A, B\}$ where $e(A, B) \ge e(G)/4 \ge \eps n(n-1)/8$.
    
    We now consider the bipartite graph that consists of edges of $G$ going from $A$ to $B$, from which we remove orientations. 
    As long as we can, we remove a vertex of degree less than $\frac{\eps n}{8}$ in the current subgraph.
    Denote the sets of vertices remaining in $A$ and $B$ by $X$ and $Y$, respectively. We note that less than $(n-1) \cdot \frac{\eps n}{8} \le e(A,B)$ edges were removed in this process (since after $n-1$ steps we are left with only a single vertex, so no edges), hence both $X$ and $Y$ are non-empty. It follows that $(X, Y)$ is an $\frac{\eps n}{8}$-mindegree pair. 
\end{proof}

Let $\overrightarrow{P}_{n_1, \ldots, n_t}$ denote an oriented path consisting of $t$ blocks of maximal directed subpaths, the $i$-th of which has order $n_i\ge 2$ (see Figure \ref{fig:dirpath}). The following lemma is one of the main tools that we use in our proofs of Theorems \ref{thm:tournament-trees} and \ref{thm:tournament-path} (recall that $l(P)$ is the length of the longest directed subpath of $P$). 
\begin{lem}\label{lem:path-vs-indep}
    Let $G$ be an oriented graph and let $P$ be an oriented path of length $n$, with $l=l(P)$. If $G$ has a $k$-mindegree pair, then either $P$ is a subgraph of $G$, or $G$ contains an independent set of size at least $\frac{k-n}{l}$.
\end{lem}

\begin{proof}
    Assume, for the sake of contradiction, that $G$ contains no independent set of size at least $\frac{k-n}{l}$. For every subset $S$ of more than $k-n$ vertices of $G$, such that the chromatic number of the induced subgraph $G[S]$ is at least $|S|/ \frac{k-n}{l} > l,$ Theorem \ref{thm:ghrv} implies that  $G[S]$ contains a directed path of length $l$.
    
    Let $P=\overrightarrow{P}_{n_1, \ldots, n_t}$; without loss of generality, we assume that the first edge of $P$ is directed away from the first vertex of $P$. 
    
    \begin{figure}[h]
        \caption{$\protect\overrightarrow{P}_{4,3,4}$.}
        \includegraphics{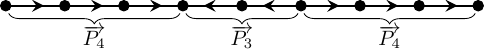}
        \label{fig:dirpath}
	\end{figure}
    
    Let $(X, Y)$ be a $k$-mindegree pair in $G$, and let $u_0 \in X$. Since $u_0$ has at least $k$ out-neighbours in $Y$ and $n_1-1 \le l$, we can find a $\dirpath{n_1-1}$ within the out-neighbourhood of $u_0$ in $Y$, which together with $u_0$ forms a $\dirpath{n_1}$. Let $u_1$ be the last vertex of this path. 
    Continuing this process for $i$ steps, suppose that we already embedded $\overrightarrow{P}_{n_1, \ldots, n_{i}}$ such that the last vertex $u_i$ is in $X$ if $i$ is even, and in $Y$ if $i$ is odd. We assume that $i$ is even, so $u_i \in X$; the case where $i$ is odd can be treated similarly. 
    Let $S$ be the set of out-neighbours of $u_i$ in $Y$ that do not belong to $\overrightarrow{P}_{n_1, \ldots, n_{i}}$. Then, by the $k$-mindegree pair assumption $|S| > k-n$, as at most $n-1$ vertices of $Y$ belong to $\overrightarrow{P}_{n_1, \ldots, n_{i}}$.
    Thus, as explained above, $S$ contains a $\dirpath{n_{i+1}-1}$. By connecting $u_i$ to the first vertex of this path we obtain a copy of $\overrightarrow{P}_{n_1,\ldots,n_{i+1}}$ whose last vertex is in $Y$ (see Figure \ref{fig:left-right} for an illustration of this process). By continuing this process until $i = t$ we obtain a copy of $P$, as required. 
    \end{proof}
\begin{figure}[h]
	\caption{The setting of the above argument.}
	\includegraphics{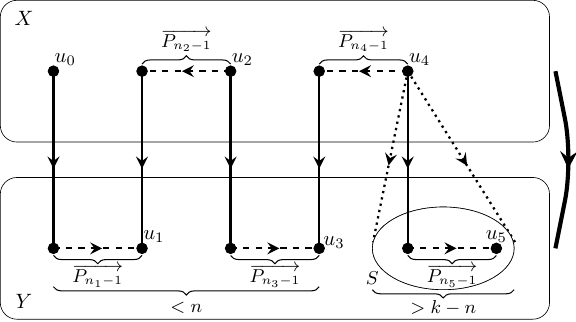}
	\label{fig:left-right}
\end{figure}

\subsection{Oriented paths}
In this subsection we prove Theorem \ref{thm:tournament-path}.
\thmpath*
\begin{proof}    
    We start with the lower bound. 
    Let $T$ be the tournament on vertex set $[l]^{k-1} \times [n]$ with edges oriented according to the lexicographic order, i.e.~if $x = (x_i)_{i \in [k]}$ and $y = (y_i)_{i \in [k]}$ are distinct vertices, and $i$ is the first coordinate in which they differ, then $xy$ is an edge in $T$ if and only if $x_i < y_i$ (so, $T$ is a transitive tournament). We colour an edge $xy$ with colour $i \in [k]$ if the first coordinate in which $x$ and $y$ differ is $i$. Note that edges of colour $k$ form a disjoint union of $n$-vertex tournaments, hence there is no copy of $P$ in colour $k$.
    Moreover, given a directed path whose edges have colour $i < k$ then the $i$-th coordinates of the vertices of the path are increasing, hence there is no $i$-coloured directed path of length $l$. It follows that there is no monochromatic copy of $P$.

    \begin{figure}[h]
        \caption{A depiction of the above example for $k=3,l=2,n=3$.}
        \includegraphics{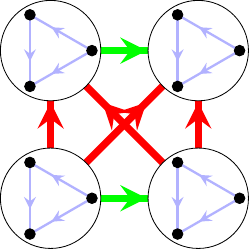}
        \label{fig:example}
    \end{figure}
    
    For the upper bound, we proceed by induction on $k$, proving the result with $a_k = (8^{k}-2) \cdot k!$ in place of $c_k$. If $k=1$, the statement follows from Theorem \ref{thm:el-sahili} (and in fact from an earlier, slightly better result which only applies for paths in \cite{reid83}). Assuming $k \ge 2$ and that the statement holds for $k-1$ we proceed.
    
    Let $m = a_k \cdot n \cdot l^{k-1},$ and let $T$ be a $k$-coloured $m$-vertex tournament.
    Let $c$ be the majority colour, and let $T_c$ be the subgraph of $T$ consisting of edges in $T$ whose colour is $c$. Then $e(T_c) \ge \frac{1}{k}\binom{m}{2}$. 
    By Lemma \ref{lem:split}, with $\eps= \frac{1}{k}$, $T_c$ contains an $\frac{m}{8k}$-mindegree pair.
    
    By Lemma \ref{lem:path-vs-indep}, $T_c$ either has a copy of $P$, in which case we are done, or it contains an independent set of size at least
    $$\frac{\frac{m}{8k}-n}{l} \, \ge \, 
    \frac{\frac{m}{8k} - l^{k-1}n}{l} \, = \, 
    \left(\frac{a_k}{8k} - 1\right)l^{k-2}n \, \ge \,
    a_{k-1} l^{k-2}n.$$
    This independent set corresponds to a subtournament $T'$ of $T$ whose edges avoid colour $c$, so they use only $k-1$ colours.
    Hence, by induction, $T'$ contains a monochromatic copy of $P$, as required.        
\end{proof}

Our argument easily extends to prove the following asymmetric version of Theorem \ref{thm:tournament-path}.
\begin{thm} \label{thm:tournament-path-m}
    Let $P_1, \ldots, P_k$ be oriented paths, and denote $l_i = l(P_i)$ and $L_i = |P_i| - 1$. Then the following holds, where $c_k > 0$ is a constant that depends only on $k$.
    $$l_1 \cdot \ldots \cdot l_k \cdot \max_{i \in [k]} \left\{ \frac{L_i}{l_i} \right\}
    \le \RT(P_1,\ldots, P_k) 
    \le c_k \cdot l_1 \cdot \ldots \cdot l_k \cdot \max_{i \in [k]}\left\{ \frac{L_i}{l_i} \right\}.$$    
\end{thm}

\subsection{Trees}
In this subsection we prove Theorem \ref{thm:tournament-trees}. 
We shall prove the following lemma; taking $l = 0$ and $T_1 = \ldots = T_k = T$, Theorem \ref{thm:tournament-trees} follows.
\begin{lem}\label{lem:reduced-one}
	Let $0 \le l \le k$. There exists a function $f(k, l)$ such that the following holds for any collection of oriented trees $T_1,\ldots, T_k$,  
    \begin{equation}\label{eq:ind-ass}
    \RT(T_1, \ldots, T_k) \le f_{k,l} \big( \leaf(T_1), \ldots, \leaf(T_{l}) \big) \cdot |T_1| \cdots |T_k|.
    \end{equation}
    
\end{lem}

\begin{proof}
    Let $a = a_k = (8k)^k$; we will prove the statement of Lemma \ref{lem:reduced-one} with $f_{k,l}(x_1, \ldots, x_l)=(2a)^{x_1+\ldots+x_l+(k-l)(2a+1)}.$
    The proof is by induction over triples $\big(k, k-l, \leaf(T_1) + \ldots + \leaf(T_{l})+|T_{l+1}|+ \ldots +|T_k| \big)$ ordered lexicographically. We denote by $A_{k,l}(T_1,\ldots,T_k)$ the inductive claim (\ref{eq:ind-ass}), with parameters $k,l,T_i$. 
        
    As the basis, we note that for $k=1$ the result follows from Theorem \ref{thm:el-sahili} which gives $\RT(T_1) \le 3|T_1|$. We now assume $k\ge 2$. For the step of the induction we will prove $A_{k,l}(T_1,\ldots,T_k)$ while assuming that the following inequalities holds: $A_{k-1,l'}(T_1',\ldots,T_{k-1}')$ for any $l' \le k-1$ and any trees $T_i'$; $A_{k,l+1}(T_1',\ldots,T_{k}')$ for any trees $T_i',$ unless $l = k$; and $A_{k,l}(T_1',\ldots,T_k')$ for any trees $T_i'$ satisfying $\leaf(T_1') + \ldots + \leaf(T_{l}') + |T_{l+1}'|+ \ldots +|T_k'| < \leaf(T_1) + \ldots + \leaf(T_{l}) + |T_{l+1}| + \ldots + |T_k|$.
    
    
    We note that when any $T_i$ is a single vertex the claim is trivial, so we assume $\leaf(T_i) \ge 2$ for all $i$. 
    
    Let $n= f_{k,l} \big( \leaf(T_1), \ldots, \leaf(T_{l}) \big) \cdot |T_1| \cdots |T_k|$ and let us assume, for the sake of contradiction, that $T$ is a tournament on $n$ vertices, whose edges are $k$-coloured such that there is no $T_i$ in colour $i$ for any $i$.
         
	Consider the following process, which finds a sequence of colours $c_i$ and a sequence of pairs $(X_i, Y_i)$ of sets of vertices (for convenience, let $X_0$ denote the set of vertices in the tournament).
	Suppose that $c_j$ and $(X_j, Y_j)$ have been defined for $j < i$. Let $c_i$ be a majority colour in the tournament induced by $X_{i-1}$; by Lemma \ref{lem:split}, we may pick a $\frac{|X_{i-1}|}{8k}$-mindegree pair $(X_i, Y_i)$ in the subgraph of $X_{i-1}$ consisting of all edges in colour $c_i$. If $c_i \le l$, or if $c_i = c_j$ for some $j < i$, we stop. Note that this process will go on for at most $k-l+1$ iterations; denote by $i$ the number of steps taken until the process was stopped. Additionally, note that $|X_j| \ge \frac{n}{(8k)^j}$ for $j \le i$.
	If $c_i \le l$, we may assume, without loss of generality, that $c_i = 1$, which leads us to Case 1 below. Otherwise, we have $c_i = c_j > l$, for some $j < i$; without loss of generality, $c_i = k$, as in Case 2 below.
    \vspace{-0.3cm}
    \subsubsection*{Case 1. we found an $\frac{n}{(8k)^k}$-mindegree pair $(X_i, Y_i)$ for colour $1$.} 
        Recall that $a=(8k)^k$, so $(X_i,Y_i)$ is an $\frac{n}{a}$-mindegree pair.
    
        If $\leaf(T_1) > 2$, let $v$ be a leaf of $T_1$ and let $u$ be the closest vertex to $v$ which has degree at least $3$ in the underlying graph. In particular, $v$ is joined to $u$ by a path $P$; let $T_1'$ be the subtree obtained from $T_1$ by removing the vertices of $P$, except for $u$. Then  $\leaf(T_1') = \leaf(T_1)-1$, as $u$ does not become a leaf upon removal of $P - u$. Without loss of generality, let us assume that the first edge (looking from $u$ to $v$) of $P$ is directed away from $u$.
                 
	We have 
	\begin{align*}
        		|X_i| \ge \frac{n}{a} 
        		& = 2 \cdot (2a)^{\leaf(T_1) - 1 + \ldots + \leaf(T_l) + (k-l)(2a+1)} |T_1| \cdots  |T_k|  \\
	        	& > (2a)^{\leaf(T_1') + \ldots + \leaf(T_l) + (k-l)(2a+1)} |T_1'| \cdots  |T_k|.
	\end{align*}
	Hence, using $A_{k,l}(T_1',T_2\ldots,T_k)$ and our assumption that there is no copy of $T_i$ in colour $i,$ for $i > 1$, there is a copy of $T_1'$ in colour $1$ within $X_i$.

	Let $u'\in X_i$ be the vertex corresponding to $u$ in this embedding of $T_1'$, then $u'$ has at least $\frac{n}{a}$ out-neighbours in $Y_i$. So applying $A_{k,l}(P-u, T_2,\ldots,T_k)$ within this neighbourhood, as there is no $T_i$ in colour $i$, we can find $P-u$ in colour $1$ within this neighbourhood (see Figure \ref{fig:X-Y}). The inductive assumption applies as $\leaf(P-u)=2<\leaf(T_1)$ and $\frac{n}{a} \ge f_{k,l}\left(2,\leaf(T_2), \ldots, \leaf(T_l)\right)|P-u| \cdot |T_2| \cdots |T_k|$. Using the out-edge from $u'$ to the first vertex of the embedded $P-u$ and appending $P-u$ we find a copy of $T_1$ in colour $1$, a contradiction.
	
	\begin{figure}[h]
		\caption{The setting of the argument in Case 1.}
		\includegraphics{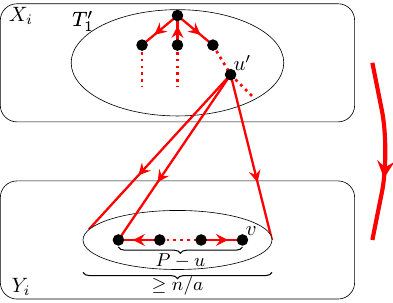}
		\label{fig:X-Y}
	\end{figure}
         
	If $\leaf(T_1) = 2,$ then $T_1$ is a path, so by Lemma \ref{lem:path-vs-indep} we can find an independent set $S$, with respect to colour $1$, of size at least the following.
	\begin{align*}
		\frac{  \frac{n}{a}-|T_1|}{|T_1|} 
		& =  2 \cdot (2a)^{\leaf(T_1) - 1 + \ldots + \leaf(T_l) + (k-l)(2a+1)}|T_2| \cdots |T_k|-1 \\
		& \ge (2a)^{\leaf(T_2) + \ldots + \leaf(T_l) + (k-l)(2a+1)}|T_2| \cdots |T_k|.
	\end{align*}
        
	By $A_{k-1,l-1}(T_2,\ldots,T_k)$, applied to the set $S$, we find a copy of $T_l$ in colour $l$ for some $2 \le l \le k$, a contradiction to our assumption.
        
    \vspace{-0.3cm}    
    \subsubsection*{Case 2. we found $i<j$ such that $c_i=c_j=k$.}  
	Let $X=X_j, Y=Y_j, Z=Y_i$, then because $(X_i,Y_i)$ and $(X_j,Y_j)$ are $\frac{n}{a}$-mindegree pairs for edges of colour $k$, every vertex in $Y$ has at least $n/a$ in-neighbours in $X$ and at least $n/a$ out-neighbours in $Z$ (since $Y \subseteq X_i$ and $(X_i,Y_i)$ is a $\frac{n}{a}$-mindegree pair), with respect to colour $k$.

	We consider the $2a$-core of $T_{k}$ (with respect to an arbitrary root $r$ of $T_k$; see Definition \ref{def:core}), which we denote by $T_{k}'$. We know that $\leaf(T_{k}') \le 2a$, so the following holds.
	\begin{align*}
		\frac{n}{a}  & = 2 \cdot (2a)^{\leaf(T_1) + \ldots + \leaf(T_l) + (k-l)(2a+1)-1} \cdot |T_1| \cdots |T_k| \\
        & > (2a)^{\leaf(T_1) + \ldots + \leaf(T_l) + 2a + (k - l - 1)(2a+1)} \cdot |T_1| \cdots |T_{k}| \\
        & \ge (2a)^{\leaf(T_1) + \ldots + \leaf(T_l) + \leaf(T_k') + (k-l-1)(2a+1)} \cdot |T_1| \cdots |T_{k-1}| \cdot |T_{k}'|.
	\end{align*} 
	Hence, we can find $T_{k}'$ in colour $k$ within $Y$, using $A_{k,l+1}(T_1, \ldots, T_l, T_k')$ (note that we have $l<k$ in this case, as when $k=l$ we automatically end up in Case 1). 
         
	Within a subset of vertices of size $t \ge \frac{n}{2a}$, we can always find, in colour $k,$ any tree $T_k''$ of order up to $\frac{|T_{k}|}{2a}$. This follows from $A_{k,l}(T_1,\ldots,T_{k-1},T_k'')$, which in turn follows as 
	\begin{align*}
		t \ge \frac{n}{2a} 
		& = f_{k,l}\big(\leaf(T_1), \ldots, \leaf(T_l)\big)\cdot |T_1|\cdots |T_{k-1}| \cdot \frac{|T_{k}|}{2a} \\
       	& \ge f_{k,l}\big(\leaf(T_1), \ldots, \leaf(T_l)\big)\cdot|T_1| \cdots |T_{k-1}| \cdot |T_{k}''|.
	\end{align*}       	
        
	Now given a vertex $x$ of $T_{k}'$, we can embed its children (in $T_k$, with respect to $r$), not already in $T_k'$, together with their subtrees in $T_k$, as follows. Let $y$ be an out-child of $x$ which is not in $T_k'$, and denote its subtree in $T_k$ by $S$. Consider the set $U$ of out-neighbours of $x$ in $Z$ (with respect to colour $k$) which were not already used. Then $|U| \ge \frac{n}{a} - |T_k| \ge \frac{n}{2a}$, hence we may find a copy of $S$ in colour $k,$ in $U$, as $|S| \le |T_k|/2a$ by the definition of a $2a$-core (see Figure \ref{fig:X-Y-Z}). If $y$ is an in-child of $x$, one can similarly embed its subtree in colour $k$ within the in-neighbourhood of $x$ in $X$. We thus obtain a copy of $T_k$ of colour $k$, a contradiction to our assumptions.
\end{proof}

\begin{figure}[h]
    \caption{The setting of the argument in Case 2.}
    \includegraphics{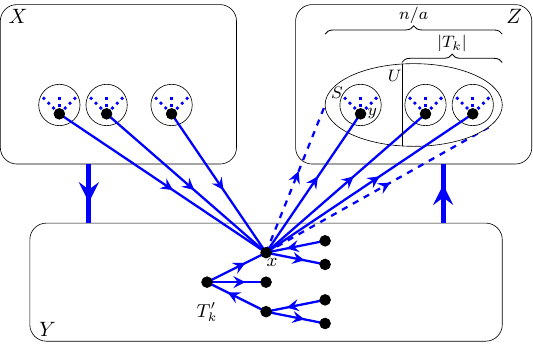}
    \label{fig:X-Y-Z}
\end{figure}

    

\section{Ramsey number on the complete directed graph}\label{sec-compl}

We start by defining the asymmetric directed Ramsey numbers; let $\R(G_1,G_2, \ldots, G_k)$ be the smallest integer $n$ such that in any $k$-colouring of $\dircomp{n}$ there is an $i$ such that there is a copy of $G_i$ in colour $i$. The main result of this section is the following generalisation of Theorem \ref{thm:complete-trees} to asymmetric directed Ramsey numbers.

\begin{thm}\label{thm:complete-trees-m}
    Let $k \ge 2$, there is a constant $c_k$ such that for any oriented trees $T_1, \ldots, T_k$, we have     
    $$\R(T_1,\ldots, T_k) \le c_k|T_1| \cdots |T_k|\left(\sum_{i=1}^k \frac{1}{|T_i|}\right).$$
\end{thm}


We say that a tree $T$ is \textit{out-directed} if there is a vertex $r$, which we call the \emph{root}, such that all edges of $T$ are directed away from $r$. Similarly, $T$ is an \emph{in-directed} tree if all its edges are directed towards a certain vertex $r$ (see Figure \ref{fig:in-vs-out}).
\begin{figure}[h]
\caption{An out-directed tree and an in-directed tree.}
\includegraphics[]{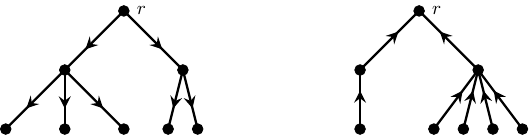}
\label{fig:in-vs-out}
\end{figure}

Our proof strategy goes as follows.
\begin{enumerate}
	\item In the following subsection, \ref{sec-prel}, we state and prove several preliminaries that we shall use throughout the proof. An important part of the preliminaries is a lemma that shows that, by paying a constant factor, we can focus our attention on out- or in-directed trees.

    \item 
    In Subsection \ref{sec-path-vs-tree}, we prove the very special case of Theorem \ref{thm:complete-trees-m}, of a directed path vs.\ an in- or out-directed tree. This requires considerable effort, especially when the path is fairly long in comparison with the tree, in which case we first find a cut with no red edges (where red is the colour of the required path), and then embed the tree in blue.
    Using the preliminary result mentioned above, this extends easily to the case when $T_1$ is an arbitrary oriented path and $T_2$ an arbitrary oriented tree.

    \item In the subsequent subsection, \ref{sec-general-case}, we prove Theorem \ref{thm:complete-trees-m} for any two trees $T_1$ and $T_2$. We do so by first proving that $\R(T_1, T_2) \le c(|T_1| + |T_2|)$ while allowing $c$ to depend on the number of leaves of $T_1$. Then,  considering cores and using some simple properties of the structure of a potential colouring, we remove the dependency on the number of leaves.
	\item In the final subsection, \ref{sec-more-col}, we show how to extend the result to more colours, thus completing the proof of Theorem \ref{thm:complete-trees-m}. The proof here follows an argument which is morally similar to the corresponding proof in the oriented case, but we need an additional tool, namely the Ramsey number of an out-tree with constantly many leaves vs.\ a complete graph.
\end{enumerate}
\subsection{Preliminaries}\label{sec-prel}

We will use the following result of Gy\'arf\'as and Lehel \cite{gyarfas1973ramsey} and Williamson \cite{williamson73}, which finds the exact Ramsey number for two directed paths, in several places in our proof, although our methods would give an independent argument for a linear bound on this Ramsey number. Note that this is a special case of Theorem \ref{thm:complete-trees-m} for two colours and directed paths.
\begin{thm} [Gy\'arf\'as, Lehel \cite{gyarfas1973ramsey}, Williamson \cite{williamson73}] \label{thm:williamson}
		$\R ( \dirpath{n},\dirpath{m} ) 
		= \left\{ 
		\begin{array}{ll}
			n + m - 3 & n, m \ge 3 \\
			n + m - 2 & n=2 \text{ or } m=2 \\
			1 & n=1 \text{ or } m=1.
		\end{array}
		\right.$
\end{thm}

The following simple embedding lemmas will prove surprisingly useful in controlling the number of vertices with small out or in-degree in a single colour.
\begin{lem}\label{lem:bidirect-embed}
Let $\eps>0,n\ge 2$, given a colouring of $\dircomp{n}$ in which there are at least $(1+\eps)\binom{n}{2}$ blue edges, then there exists any directed tree of order at most $\lceil\frac{\eps n}{2}\rceil$ in blue.
\end{lem}
\begin{proof}
    As there are $\binom{n}{2}$ pairs of vertices, there are at least $\eps \binom{n}{2}$ bidirected blue edges. Now remove, one by one, vertices that are incident with fewer than $\eps n/2$ bidirected blue edges. Since we remove fewer than $(n-1)\eps n/2$ edges in this process, we remain with a non-empty set of vertices in which the minimum blue bidirected degree is at least $\lceil\eps n/2\rceil$. It is now easy to see that any blue directed tree on at most $\lceil\eps n/2\rceil$ vertices can be embedded.
\end{proof}
We will only use the ceiling in the above lemma to gain a tight result in the following lemma.
\begin{lem}\label{lem:high-deg-embed}
    Given a $2$-colouring of a complete digraph, if there are at least $2k + 2l$ vertices with at most $k$ red out-neighbours, then there is a blue copy of any directed tree on at most $l$ vertices.
\end{lem}

\begin{proof} 
    Suppose that $S$ is a set of $2k + 2l$ vertices with red out-degree at most $k$. Then the number of red edges in $S$ is at most $k |S|$, hence the number of blue edges in $S$ is at least $|S|(|S|-1) - k|S| = \left(2 - \frac{2k}{|S|-1}\right)\binom{|S|}{2}$. By Lemma \ref{lem:bidirect-embed}, there is a blue copy of any tree of order at most $\left\lceil \left(1 - \frac{2k}{|S|-1}\right)\frac{|S|}{2} \right \rceil \ge l$, where the inequality follows as $\left(1 - \frac{2k}{|S|-1}\right)\frac{|S|}{2}= k+l - k - \frac{k}{|S|-1} > l - 1$, as required.
\end{proof}

The following lemma, based on the Depth-first-search (DFS) algorithm, is an easy generalisation, for trees, of the version for paths introduced by Ben-Eliezer, Krivelevich and Sudakov \cite{ben2012long-cycles,ben2012size}; it will be very useful in several parts of this section.

\begin{lem}[DFS] \label{lem:dfs}
    Let $G$ be a directed graph and $T$ an out-directed tree, which is not a subgraph of $G$. Then there is a partition $\{U, X, Y\}$ of $V(G)$  such that $U$ is the vertex set of a subgraph of $G$ which is isomorphic to a subtree of $T$, $|X| = |Y|$, and every vertex in $X$ has fewer than $|T|$ out-neighbours in $Y$.
    
\end{lem}

\begin{proof}
    We construct the required $X$, $Y$ and $U$ using the following process.
    We maintain a partition $\{U, X, Y\}$ of $V(G)$ and an embedding $T'$ of a subtree of $T$ in $G$, with vertex set $U$. 
    We keep the following properties invariant. 
    \begin{enumerate}
    \item[(P1)] \label{itm:y-ge-x}
    		$|Y| \ge |X|.$
    \item[(P2)] \label{itm:U-subtree}
    		$T'$ is a subgraph of $G$, which is isomorphic to a subtree of $T$ rooted at the same root, such that every vertex of $T'$ either has the same number of children as in $T$ or is an \emph{out-leaf} (a leaf with out-degree $0$) of $T'.$
    \item[(P3)] \label{itm:x-few-outneighs}
    		Any vertex in $X$ has fewer than $|T|$ out-neighbours in $Y$.
    \end{enumerate}  
    We start with $X=\emptyset$, $Y=V$ and $U = \emptyset$. 
    At each step, if $|X|=|Y|$ we stop the process; otherwise we perform the following procedure.
    \begin{enumerate}
        \item \label{itm:1} If $T'$ is empty, we remove an arbitrary vertex $u$ from $Y$, and put it in $U$, thus letting $T'$ be a single vertex that corresponds to the root of $T$.
        \item \label{itm:2} Otherwise, we pick an out-leaf $v$ of $T'$, which does not correspond to a leaf of $T$; denote its out-degree in $T$ by $d$. 
        \begin{enumerate}[label=\alph*),ref=2\alph*]
            \item \label{itm:2a} If $v$ has at least $d$ out-neighbours in $Y$ we append them to $T'$ via $v$ and remove them from $Y$, unless this would cause $|Y|$ to be smaller than $|X|$, in which case we only embed as many as needed to make $|X|=|Y|$ and stop.
            \item \label{itm:2b} Otherwise, $v$ has at most $d-1$ out-neighbours in $Y$, so we add $v$ to $X$ and put every vertex of $U$ back to $Y$.        
        \end{enumerate}
    \end{enumerate}  
    
    We note that (P1) is preserved throughout this process. Property (P2) also holds throughout the process, with a possible exception at the very end of the process (if case \ref{itm:2a} occurs at the end, and removing $d$ vertices from $Y$ would cause a violation of (P1)), where $U$ is still the vertex set of a subtree $T'$ of $T$, but one of the vertices may not be an out-leaf or have all its children present. Finally, it is easy to see that (P3) is preserved, as a vertex $x$ is moved to $X$ only in Case \ref{itm:2b}, in which case it has at most $|T'| + d - 1 < |T|$ out-neighbours in $Y$, where $T'$ is a subtree of $T$ in which $x$ is present but none of its children is, and $d$ is the out-degree of $x$ in $T$.
    
	Finally, we note that in each step $|X|$ does not decrease, and either $|X|$ increases by $1$, or $|Y|-|X|$ decreases. Thus the process cannot run indefinitely, and once it terminates, we obtain a partition $\{U, X, Y\}$ with the required properties.
	\end{proof}

Given oriented trees $T_1, \ldots, T_k$, let $S_i(T_1, \ldots, T_k)$ denote the out- or in-directed subtree $S$ of $T_i$ that maximises $\R(T_1,\ldots,T_{i-1}, S, T_{i+1}, \ldots,T_k).$ The following results allow us to focus our attention on out- or in-directed trees.

\begin{lem}\label{lem:reorient-one}
Let $\eps > 0$, let $T_1, \ldots, T_k$ be oriented trees and let $S_j=S_j(T_1, \ldots, T_k)$. 
Let $G$ be a $k$-colouring of $\dircomp{n}$ such that $n \ge 8 \eps^{-1} \left( \R(T_1,\ldots,T_{j-1},S_j,T_{j+1},\ldots, T_k)+|T_j| \right)$ and there are at least $\eps \binom{n}{2}$ edges of colour $j$. Then there is a copy of $T_i$ in colour $i$, for some $i \in [k]$.
\end{lem}

\begin{proof}
	Without loss of generality, we assume $j=1$. We assume that there is no copy of $T_i$ in colour $i$ for $i \ge 2$, and will deduce that there is a copy of $T_1$ in colour $1$.

	Consider the graph of colour $1$ edges; it has at least $\eps \binom{n}{2}$ edges, so by Lemma \ref{lem:split} we can find an $\eps n/8$-mindegree pair $(X,Y)$ in it. 

	We select a vertex $r$ of $T_1$. Let $U_0$ consist of all vertices of $T_1$ reachable by out-directed paths starting in $r$. For $i \ge 1$, let $U_{i}$ consist of vertices not in any $U_{j},$ for $j\le i-1$, which are reachable by out-directed paths starting in $U_{i-1}$ if $i$ is even, or in-directed paths if $i$ is odd. We note that for even $i$, each $T_1[U_{i}]$ is a forest of out-directed subtrees of $T_1$, whose roots are out-neighbours of vertices in $U_{i-1}$; and, when $i$ is odd, $T_1[U_{i}]$ is a forest of in-directed subtrees whose roots are in-neighbours of vertices in $U_{i-1}$ (see Figure \ref{fig:in-out-split}).   
    \begin{figure}[h]
\caption{An illustration of the definition of $U_i$.}
\includegraphics[]{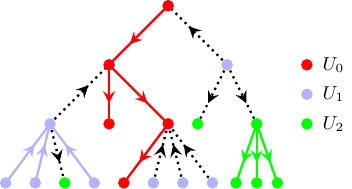}
\label{fig:in-out-split}
\end{figure}
\begin{claim} \label{claim:embed}
	There is a copy of $T_1[U_0 \cup \ldots \cup U_t]$ in colour $1$, such that $U_t$ is embedded within $Y$ if $t$ is even and within $X$ if it is odd.
\end{claim}
\begin{proof}
	We show this inductively. For the basis, we can embed $U_0$ in colour $1$ within $Y$ as $|Y| \ge \frac{\eps n}{8} \ge \R(S_1,T_2,\ldots, T_k)+|T_1| \ge \R(T_1[U_0],T_2,\ldots, T_k)$, and $T_1[U_0]$ is an out-directed subtree of $T_1$, and there are no copies of $T_i$ in colour $i$, for $i \ge 2$. Suppose that the claim holds for $i-1$, and suppose that $i$ is even; the case where $i$ is odd can be treated similarly. 
	Consider an embedding of $T_1[U_0 \cup \ldots \cup U_{i-1}]$ where $U_{i-1}$ is embedded within $X$. For each tree $S$ in $T_1[U_i]$, we need to embed it within the out-neighbourhood of its parent vertex $u$ in $U_{i-1}$ which is already embedded in $X$. Note that every vertex in $X$ has out-neighbourhood in $Y$ of size at least $\frac{\eps n}{8} \ge \R(S_1,T_2,\ldots, T_k)+|T_1|$, and so far we used at most $|T_1|$ vertices, hence we can find a copy of $S$ in colour $1$ within the out-neighbourhood of $u$ in $Y$, avoiding used vertices. We do this one by one for each subtree of $T$ in $T_1[U_i]$, thus completing an embedding of $T_1[U_0 \cup \ldots \cup U_i]$ with the required properties. This completes the inductive claim.
\end{proof}
Applying Claim \ref{claim:embed} for the largest $t$ such that $U_t$ is non-empty shows there is a copy $T_1$ in colour $1$, and we are done. 
\end{proof}

\begin{lem}\label{lem:reorientation}
	Given oriented trees $T_1$ and $T_2$, we have $\R(T_1, T_2) \le 16^2\max\{\R(R_1,R_2)+|T_1|+|T_2|\}$, where the maximum is taken over out- or in-directed subtrees $R_1$ and $R_2$ of $T_1$ and $T_2$ respectively.
\end{lem}

\begin{proof}
Let $S_1=S_1(T_1,T_2)$, so $S_1$ is the out or in-directed subtree $S$ of $T_1$ which maximizes $\R(S,T_{2}).$

\begin{claim}\label{claim:reorient-intermediate}
$\R(T_1, T_2) \le 16(\R(S_1,T_2)+|T_1|)$.
\end{claim}

\begin{proof}
	Let $n=16(\R(S_1,T_2)+|T_1|)$, and let $G$ be a $2$-colouring of $\dircomp{n}$.
	If there are at least $\frac{1}{2}\binom{n}{2}$ edges in red, the proof follows from Lemma \ref{lem:reorient-one} (with $k=2$, $j = 1$ and $\eps = 1/2$). Otherwise, there are at least $\frac{3}{2}\binom{n}{2}$ blue edges. Lemma \ref{lem:bidirect-embed} implies that there is a blue copy of any tree of order $n/4 \ge |T_2|$; in particular, $G$ has a blue copy of $T_2$.
\end{proof}

Let $S_2=S_2(S_1,T_2)$. Applying Claim \ref{claim:reorient-intermediate} we get, $\R(S_1, T_2) \le 16(\R(S_1,S_2)+|T_2|)$. Hence,
\begin{align*}
	\R(T_1, T_2) 
	& \le 16\big( \R(S_1, T_2) + |T_1| \big) \\
	& \le 16\Big( 16\big(\R(S_1, S_2) + |T_2|\big) + |T_1|\Big) \\
	& \le 16^2 \big(\R(S_1, S_2) + |T_1| + |T_2|\big),
\end{align*}
as desired.
\end{proof}

The following corollary, which is a special case of Theorem \ref{thm:complete-trees-m}, illustrates how we can use Lemma \ref{lem:reorientation}, and will prove useful in its own right.
\begin{cor}\label{cor:oriented-paths}
Let $P$ and $Q$ be arbitrarily oriented paths. Then
$$\R(P,Q) \le 2\cdot 16^2(|P|+|Q|).$$
\end{cor}
\begin{proof}
Note that any out-~or in-directed subtree of an oriented path is a directed path. The in- or out-directed subtrees $R_1$ and $R_2$ of $P$ and $Q$ which maximize $\R(R_1,R_2)$ are directed subpaths of maximum length, so by Lemma \ref{lem:reorientation} and by Theorem \ref{thm:williamson},
\begin{align*}
	\R(P,Q) 
	& \le 16^2\Big(\R\left(\dirpath{l(P)},\dirpath{l(Q)}\right)+|P|+|Q|\Big) \\
	& \le 16^2(l(P)+l(Q)+1+|P|+|Q|) \\
	& < 2\cdot 16^2(|P|+|Q|),
\end{align*}
as desired.
\end{proof}

The following simple lemma will provide us with a way to control the degrees of an out-directed tree using the number of \emph{out-leaves}, which we define to be leaves with out-degree $0$. We remark that, as in the remainder of this section, we consider only out-leaves in order to avoid the possibility of counting the root when it is a leaf.
\begin{lem}\label{lem:deg-leaf}
Let $v_1, \ldots, v_k$ be some vertices of an out-directed rooted tree $T$ of size $n$, with $a$ out-leaves. If $v_i$ has out-degree $d_i$ then 
$$d_1+d_2+\ldots+d_k-k+1 \le a.$$
\end{lem}
\begin{proof}
Let $v_{k+1}, \ldots, v_n$ be the remaining vertices of the tree (with out-degrees marked by $d_i$). Then,
\begin{align*}
	d_1+d_2+\ldots+d_k-k+1 
	& = 1+(d_1-1)+\ldots+(d_k-1) \\
	&\le 1+(d_1-1)+\ldots +(d_n-1)+a \\
	& = 1+d_1+\ldots+d_n-n+a \\
	& = 1+n-1-n+a=a.
\end{align*} 
The first inequality follows as $d_i-1$ is non-negative, unless it comes from an out-leaf in which case it is $-1$. The penultimate equality follows since $d_1+\ldots+d_n$ is the number of edges, which is $n-1$.
\end{proof}


\subsection{Path vs.\ Tree}\label{sec-path-vs-tree}
In this subsection we prove a special case of Theorem \ref{thm:complete-trees-m} for two colours, where one of the trees is an oriented path. 
The following is the key result of this subsection.
\begin{thm}\label{thm:dirpath-vs-outtree}
   Let $T$ be an out-directed tree on $m$ vertices. Then the following holds for every $n$, where $c$ is some absolute constant.
    $$\R(\dirpath{n},T) \le c(n+m).$$
\end{thm}
\begin{proof}
	Let $c = 79$.
    We use induction on $n+m$ to prove that $\R(\dirpath{n},T) \le c(n+m-2)$ given $n+m \ge 3$, or equivalently $n>1$ or $m>1$; we require this slightly stronger statement in order to avoid separate treatment of small cases, but it requires us to make sure that at least one of the trees we are invoking induction on has order greater than one. For the basis notice that case $n=1$ or $m=1$ is trivial, hence we may assume that $n, m \ge 2$.
    
    Let $N=c(n+m-2)$. Let us assume, for the sake of contradiction, that there is a $2$-colouring of $\dircomp{N}$, whose vertex set we denote by $V$, with no red $\dirpath{n}$ or blue $T$. We consider two cases: $n \le 16m$ and $n \ge 16m$; the proof in the first case is relatively simple, whereas the proof in the second case requires significant effort.
    
    \vspace{-0.3cm}
    \subsubsection*{Case 1. $n \le 16m$.}
        
        Let $T'$ be the $2$-core of $T$ (see Definition \ref{def:core}), so $T'$ is a directed path. Let $S$ be the set of vertices with blue out-degree at least $c(n + m/2 - 2) + m$. If $|S|\ge n+m$, then, as $\R(\dirpath{n},\dirpath{m}) \le n+m-2$ (see Theorem \ref{thm:williamson}), we can find a blue $T'$ in $S$. We now show that it is possible to embed the trees of the forest $T \setminus V(T')$ one by one in the blue out-neighbourhood of the corresponding parent in $T'$. Indeed, let $R$ be such a tree, and let $u \in T'$ be the parent of the root of $R$. Then $u$ has at least $c(n + m/2 - 2)$ unused blue out-neighbours, hence by induction (recall that $n \ge 2$, so we can use induction), and by the assumption that there is no red $\dirpath{n}$, we find a blue $R$. Having done this for all the trees in $T \setminus V(T')$, we obtain a blue $T$, a contradiction.
    
        We are left with the case where $|S| < n+m$. Every vertex not in $S$ has red out-degree at least $N - 1 - (c(n + m/2 - 2) + m) = cm/2 - m - 1$. So, in the graph induced by $V \setminus S$ every vertex has red out-degree at least $ cm/2 - m - 1 - (n + m - 1) = cm/2 - 2m - n \ge (c/2 - 2)n/16 - n \ge n$ where we used $n \le 16m$ and $c \ge 68$. But this implies that there is a red $\dirpath{n}$, a contradiction. This completes the proof of Theorem \ref{thm:dirpath-vs-outtree} in the first case.
    
	\vspace{-0.3cm}
	\subsubsection*{Case 2. $n \ge 16m$.}
    
	We start by proving the following claim which gives us a handle on the structure of the colouring.
		\begin{claim}\label{claim:partition}
            	There is a partition $\{U, W\}$ of $V$, such that there are no red edges from $U$ to $W$, and $|U|,|W| \ge N/5$.
		\end{claim}
    
        	\begin{proof}
            	Let $A$ be the set of vertices with red out-degree at least $3N/4$. If $|A| \le N/5$ then in $V \setminus A$ every vertex has blue out-degree at least $N/4-N/5=N/20=c(n+m-2)/20 \ge c\cdot 15m/20 >m$. Hence, we can embed a blue $T$, a contradiction. It follows that $|A| \ge N/5$.
            
            	Let $B$ be the set of vertices with red in-degree at least $N/3$. By Lemma \ref{lem:dfs} (DFS), applied to the graph of blue edges, as there is no blue $T$ and $|T| = m$, we can find two disjoint sets $X$ and $Y$ satisfying $|X|=|Y| \ge N/2-m/2$, such that there are at most $m|X|$ blue edges from $X$ to $Y$. Note that every vertex in $Y \setminus B$ has at least $|X|-N/3$ blue in-neighbours in $X$. Hence, by double counting blue edges from $X$ to $Y$, we obtain the following inequality.
            	$$m|X| \ge (|Y|-|B|)\cdot (|X|-N/3).$$
            	Since $N = c(n + m - 2) \ge 15m$, we have			
            	\begin{align*}
 	         	|B| 
 	         	& \ge |X|\left(1-\frac{m}{|X|-N/3}\right) \\
 	         	& \ge (N/2-m/2)\left(1-\frac{m}{N/6-m/2}\right) \\
 	         	& \ge \frac{2N}{5}\left(1-\frac{m}{5m/2-m/2}\right) = \frac{N}{5}.
 	         \end{align*}
            
            	Suppose that there is a red path $P$ from a vertex $v$ in $B$ to a vertex $u$ in $A$; write $|P|=l$ and note that $l<n$ as there is no red $\dirpath{n}$. Let $S_v$ be the set of red in-neighbours of $v$ which are not in $P$; as $v \in B$ and $|P| \le n$, we have $|S_v| \ge N/3 - n$. Let $Q_1$ be the longest red directed path in $S_v$ and denote its number of vertices by $n_1$; note that $n_1 + l < n$, as $Q_1$ and $P$ combine to a red directed path of order $n_1 + l$. Similarly, let $S_u$ be the set of red out-neighbours of $u$ which are not in $P$ or in $Q_1$. Then, as $u \in A$, we have $|S_u| \ge 3N/4 - n$. Let $Q_2$ be the longest red path in $S_u$, and denote its number of vertices by $n_2$. Note that we can combine $Q_1$, $P$ and $Q_2$ (in this order) to form a red path of order $l + n_1 + n_2 < n$ (see Figure \ref{fig:comb-paths}). Also, by induction (which applies here as $m \ge 2$), we have 
            	\begin{align*}
				& N/3 - n \le |S_v| \le c(n_1 + m - 1), \\
				& 3N/4 - n \le |S_u| \le c(n_2 + m - 1).
			\end{align*}
			Summing up the two inequalities and recalling that $N = c(n + m - 2)$, we obtain the following.
			\begin{align*}
				13c(n + m - 2)/12 - 2n = 13N/12 - 2n \le c(n + 2m - 3).
			\end{align*}
			This, in turn, shows that 
			\begin{align*}
				(c/12 - 2)n \le 11c m/12 - 5c/6 \le 11cm / 12.
			\end{align*}
			Hence,
			\begin{align*}
				n \le \frac{11c/12}{c/12 - 2} \cdot m = \frac{11c}{c - 24} \cdot m < 16m,
			\end{align*}
			a contradiction. It follows that there are no red paths from $B$ to $A$.
            
            Let $U$ be the set of vertices that can be reached by following a red path from $B$, and let $W = V \setminus U$. By the definition of $U$, we have no red edges from  $U$ to $W$. As $B \subseteq U$ and $A \subseteq W$ we also have $|U|,|W| \ge N/5$, as desired. 
	\end{proof}
		\begin{figure}[h]
			\caption{An illustration of the construction in \Cref{claim:partition}.}
			\includegraphics[]{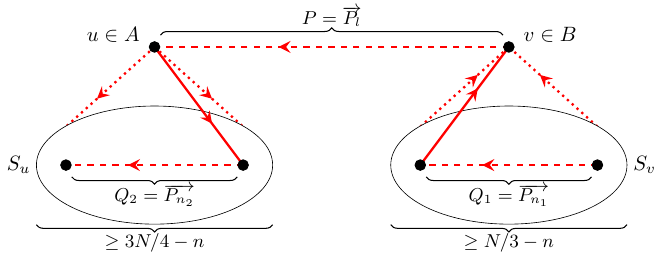}
			\label{fig:comb-paths}
	\end{figure}
    
    Let $(U,W)$ be a pair as given by the above claim. Let $a$ be the number of out-leaves of $T$, let $d_1 \ge \ldots \ge  d_{m-a}$ be the out-degrees of the vertices of $T$ with out-degree at least $1$. 
    Consider the following process, which we run for $i$ from $1$ to at most $m - a$. At step $i$, we take $w_i$ to be the vertex in $W$ whose blue out-degree towards $U$ is largest. If $d^+_U(w_i)<d_i$ we stop; otherwise, we remove $w_i$ from $W$ and remove from $U$ a set $B_i$ of $d_i$ blue out-neighbours of $w_i$ in $U$ (see Figure \ref{fig:U-W-2}).
    
    \begin{figure}[h]
        \caption{An illustration of the state of the above process at step $k$.}
        \includegraphics{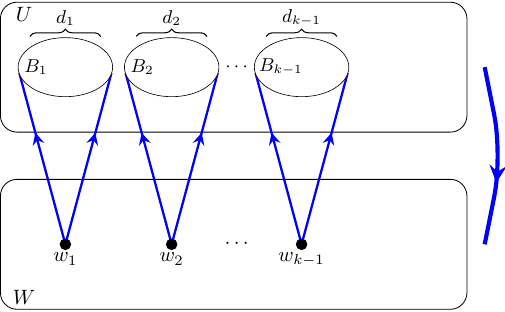}
        \label{fig:U-W-2}
        \end{figure}
            
    \begin{claim}
    If in the above process we have $d^+_U(w_i) \ge d_i$ for all $i=1,\ldots,m-a$, that is, the process does not terminate early, then there is a blue copy of $T$.
    \end{claim}
    \begin{proof}
    		Let $f$ be an injective (and hence surjective) function from the vertices of $T$ with out-degree at least $1$ to $[m - a]$, such that the out-degree of a non-out-leaf vertex $v$ is $d_{f(v)}$.
        
        We now embed the tree as follows. We say that a vertex is \textit{fully embedded} if all its children in $T$ are embedded; we say that it is \textit{partially embedded} if none of its children are embedded. We will embed $T$, such that at each stage, every embedded vertex is either fully or partially embedded. We start by embedding the root $r$ in $w_{f(r)}$. At each step we choose a partially embedded vertex $v$ and embed its children, according to the following plan.
        \begin{enumerate}
        \item \label{itm:3.1} If $v\in U$ we embed its children one by one, according to the following instructions, where $u$ is a child of $v$.
        \begin{enumerate}
            \item \label{itm:3.1a} If $u$ is not an out-leaf of $T$, we embed it in $w_{f(u)}$.
            \item \label{itm:3.1b} If $u$ is an out-leaf of $T$, then we embed it in an arbitrary unused vertex in $W$ distinct from all $w_i$.
        \end{enumerate} 
        \item \label{itm:3.2} If $v\in W$, we embed its children in $B_{f(v)}$.
        \end{enumerate}
		Note that in each stage of this process, every vertex is indeed partially or fully embedded. Furthermore,        
        	it is easy to see that the process runs until all vertices of $T$ are embedded, as each $w_i$ can only be used to embed $f^{-1}(i)$. We thus obtain a blue copy of $T$, as all edges of $T$ are embedded in edges from $U$ to $W$ (which are all blue), or in edges from some $w_i$ to $B_i$, which are also blue. This completes our proof that a blue copy of $T$ exists.
    \end{proof}
    
    By this claim we are done unless there is a $k \le m-a$ such that every vertex from $W'=W \setminus \{w_1, \ldots, w_{k-1}\}$ to $U'=U \setminus (B_1 \cup \ldots \cup B_{k-1})$ has blue out-degree smaller than $d_{k}$.
    
    Let $n_2$ be the order of a longest red path in $W'$. Denote the end of this path by $u$, and denote by $U''$ the red out-neighbourhood of $u$ in $U'$ (see Figure \ref{fig:U-W-1}). As $u$ has blue out-degree in $U'$ of at most $d_k,$ we have $|U \setminus U''|\le |U \setminus U'| + d_k=d_1+\ldots+d_k$. Let $n_1$ be the order of a longest red path in $U''$. By our assumptions, $n_1+n_2 <n$.
    \begin{figure}[h]
        \caption{An illustration of the above setting.}
        \includegraphics{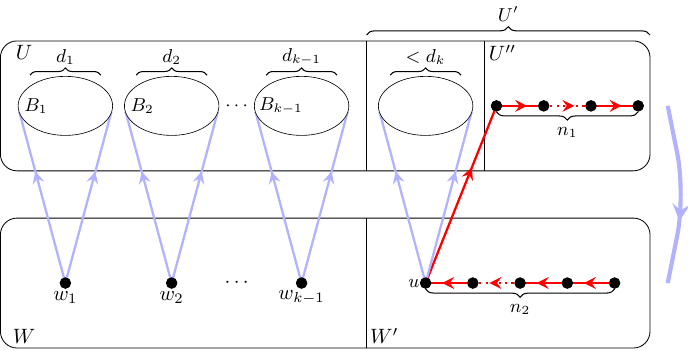}
        \label{fig:U-W-1}
        \end{figure}
    
    We now distinguish two cases, depending on how $d_1+\ldots +d_{k}$ and $k-1$ compare.
    \vspace{-0.3cm}
    \subsubsection*{Case 2a. $d_1+d_2+\ldots+d_k \le 4(k-1)$.}
            
    		We now embed $T$ in the blue graph in stages; denote by $T_1$ the currently embedded subtree of $T$. Initially, let $T_1$ be a largest subtree of $T$, rooted at $r$ (which is the root of $T$), which can be found in blue in $U''$; write $m_1=|T_1|$. We say that an embedded vertex of $T$ is \textit{incomplete} if not all its children in $T$ have already been embedded. We repeat the following. At step $j$ we take an unembedded child $v_j$ of an incomplete vertex and embed it in $w_i$, where $i\le k-1$ is the largest $i$, such that $d_i \ge d^+_T(v_j)$ and $w_i$ is not already used; note that we can always do this by the following Claim \ref{claim:majorising} applied to $i_j=f(v_j)$. We then embed the children of $v_j$ in $B_i$. Note that throughout this process, we preserve the property that all incomplete vertices are in $U$. 
    		\begin{claim}\label{claim:majorising}
    		Let $\{i_1,\cdots, i_{k-1}\} \subseteq [m-a]$. Let $g_1=\min(i_1,k-1)$, then for each $j\le k-1$ there exists an $i \le \min(i_j,k-1)$ not in $S_{j-1} := \{g_1,\ldots,g_{j-1}\}$, and we define $g_j$ to be the largest such $i$.
    		\end{claim}
    		\begin{proof}
    		We inductively show that at step $j$, there is a $t_j$ such that the following holds.
    		$$S_{j} = \bigcup_{l\le j, \,\, i_l \le k-1} \{i_l\} \:\: \bigcup \:\: \{t_j,t_j+1,\ldots,k-1\}.$$
    		For the basis, if $i_1<k-1$ then $S_1 = \{i_1\}$ so the claim holds with $t_1=k$; otherwise, $t_1=k-1$ works. We assume this holds at step $j-1$. If $i_j\ge t_{j-1}$ then there is an $l<t_{j-1}$ such that $l\notin S_{j-1}$ (as otherwise $S_{j-1}=[k-1]$), and as $g_j$ is the largest such $l$ the claim holds with $t_j=l$. If $i_j < t_{j-1}$, $i_j$ is not used as $i_j \neq i_l$ for any $l<j$, so $g_j=i_j$ and $t_j=t_{j-1}$ works. Claim \ref{claim:majorising} easily follows.
    		\end{proof}
    		
        Hence, we embedded a subtree $T_1$ of $T$ of order at least $m_1+k-1,$ with all incomplete vertices in $U$. Let $m_2$ be the largest integer such that any out-tree of order up to $m_2$ can be found in blue in $W'$. Note that if we can embed any out-tree of order up to $m_2$ we can also embed any out-directed forest of order up to $m_2$, as any such forest is a subgraph of an out-tree of the same order. We conclude that $m_1+k-1+m_2 < m$ as, otherwise, we can embed $T \setminus V(T_1)$ within $W'$.  As all edges from $U$ to $W$ are blue and all incomplete vertices of $T_1$ are in $U$, this would combine to create a blue $T$.
        
        Induction implies that $|W'|< c(n_2+m_2)$ and $|U''| < c(n_1+m_1)$ (note that $n_1 \ge 2$, as otherwise there are no red edges in $U''$ and a blue copy of $T$ could easily be found; similarly, $n_2 \ge 2$; hence, we may apply the induction hypothesis). Combining these inequalities, we get the following.
        \begin{align*}
        N = \,\, &  |U|+|W| = |U''|+|U\setminus U''|+|W'|+|W \setminus W'| \\
             < \,\, & c(n_1+m_1)+d_1+\ldots+d_k+c(n_2+m_2)+k-1 \\
             \le \, \, &  c(n_1+n_2)+c(m_1+m_2)+5(k-1) \\
             \le \, \, & c(n_1 + n_2)+c(m_1+m_2+k-1) \\
             \le \, \, & c(n - 1) + c(m - 1) = c(n+m-2)=N,
        \end{align*}        
        a contradiction.
    
    \vspace{-0.3cm}    
    \subsubsection*{Case 2b. $d_1+d_2+\ldots+d_k > 4(k-1)$.}
    
        Let $T'$ be the subtree of $T$ consisting of vertices whose subtrees have more than $a/2$ out-leaves. As each vertex can have at most one child in $T'$, $T'$ must be a directed path.
        
        As $|U''| \ge N/5-m \ge n+m$, we can find a blue $T'$ in $U''$ by Theorem \ref{thm:williamson} (path vs.\ path). Let $T''$ be the subtree of $T$ obtained by removing all its out-leaves. Let $T_1$ be a maximal subtree of $T''$, containing $T'$, which can be found in blue in $U''$. Let $m_1=|T_1|$.
        
        If $T_1=T''$ then we simply embed out-leaves of $T$ arbitrarily in $W$, and we thus find a blue $T$. Otherwise, by the inductive assumption $|U''| < c(m_1+n_1)$ (as before, $n_1 \ge 2$, so we can apply induction).
        
        Let $S_1, \ldots, S_t$ be the trees in the forest $T\setminus T_1$. By definition of $T',$ each $S_i$ has at most $\frac{a}{2}$ out-leaves. 
        If we can find a copy of $S_1$ in $W'$, we remove it from $W'$. We repeat, so at step $i$ we have removed vertex-disjoint blue copies of $S_1, \ldots, S_{i-1}$ from $W'$, and, if we can find a blue $S_i$, we remove it.
        If the process runs until $i=t$, using the fact that all edges from $U$ to $W$ are blue, we can join $T_1$ with the copies of the $S_i$'s that we found to obtain a blue copy of $T$, a contradiction. So there is an $i$, such that we cannot find a blue $S_i$ in $W'' = W' \setminus (V(S_1) \cup \ldots \cup V(S_{i-1}))$. This implies, by the inductive assumption that $|W''| < c( n_2 + |S_i| - 1)$  (the inductive assumption applies as $n_2 \ge 2$, as before). We obtain the following upper bound on $|W'|$. 
        \begin{align*}
        |W'| < \,\, &  |S_1|+\ldots+ |S_{i-1}|+cn_2+c|S_i|-c \\
             \le \,\, & c(|S_1|+\ldots+|S_t|)-(c-1)(|S_1|+\ldots+|S_{i-1}|+|S_{i+1}|+\ldots+|S_t|)+cn_2-c \\
             \le \,\, & c(m-m_1)-(c-1)a/2+cn_2-c.
        \end{align*}
        
        Where the last inequality follows as $\cup_{i=1}^t S_i$ contains all $a$ out-leaves of $T$ and $S_i$ can contain at most $a/2$ so $\cup_{j \in [t] \setminus \{i\}} S_j$ has at least $a/2$ out-leaves, so, in particular, it has at least $a/2$ vertices.
        
        By Lemma \ref{lem:deg-leaf}, we have $d_1+d_2+\ldots+d_k \le a+k-1 < a+(d_1+\ldots+d_k)/4$, which in turn implies that $d_1+\ldots+d_k < 4a/3.$
        Combining these inequalities we get the following.
        \begin{align*}
        N = \,\, |U|+|W| 
         < \,\, & d_1+\ldots+d_k+c(m_1+n_1)+k-1+c(m-m_1)-(c-1)a/2+cn_2-c \\
         < \,\, & 4a/3+a/3-(c-1)a/2+c(n_1+n_2+1)+cm-2c \\
         \le \,\, &  5a/3-(c-1)a/2+c(n+m-2) \\
         \le \,\, & c(n+m-2)=N,
        \end{align*}
        a contradiction. This completes the proof of Theorem \ref{thm:dirpath-vs-outtree}.
\end{proof}

\begin{cor}\label{cor:path-vs-tree}
Let $P$ be an oriented path and $T$ any tree, then
$$\R(P,T) \le 16^2(c+1)(|P|+|T|),$$
where $c$ is the constant given in Theorem \ref{thm:dirpath-vs-outtree}.
\end{cor}

\begin{proof}
    Theorem \ref{thm:dirpath-vs-outtree} applies for a path vs.\ in-directed tree as well by symmetry, as directed paths are both in and out-directed trees. Corollary \ref{cor:path-vs-tree} follows from Lemma \ref{lem:reorientation}.
\end{proof}

\subsection{Tree vs.\ tree}\label{sec-general-case}

In this subsection we prove the following result, which is the special case of Theorem \ref{thm:complete-trees-m} for two colours.
\begin{thm} \label{thm:complete-trees-two}
	Let $S$ and $T$ be oriented trees. Then $\R(S, T) \le c(|S| + |T|)$, where $c$ is an absolute constant.
\end{thm}

We first prove the following result, which is a special case of Theorem \ref{thm:complete-trees-two} where one of the trees has a bounded number of leaves.
\begin{lem}\label{lem:bounded-leaf-vs-any}
Let $S$ and $T$ be directed trees of orders $n$ and $m$ respectively. Then $\R(S,T) \le c_1 (n+m)$, where $c_1 = c_1(\leaf(T))$.
\end{lem}

\begin{proof}
	    We say that a vertex of a tree is \textit{branching} if it has degree at least $3$ in the underlying graph. We will need the following two simple properties of branching vertices: a tree is an oriented path if and only if it has no branching vertices; and each subtree obtained by removing a branching vertex has strictly fewer leaves than the original tree.
	    
	Let $c$ be a constant such that $\R(P, T) \le c(|P| + |T|)$ for every oriented path $P$ and every oriented tree $T$; the existence of $c$ is guaranteed by Corollary \ref{cor:path-vs-tree}.    
	Write $a_l =  c \cdot 11^{l-2}$. We show, by induction on $l$, that $\R(S, T) \le a_l(n + m)$, for any oriented tree $T$ of order $m$ and any oriented tree $S$ of order $n$ with at most $l$ leaves. The induction base, where $l = 2$, is exactly the above statement, as then $S$ must be a path. Now suppose that $l \ge 3$ and let $S$ be a tree with $l$ leaves. Then $S$ has a branching vertex $u$. Write $d = a_{l-1}(n + m)$, and consider a $2$-colouring of $\dircomp{11d}$. By Lemma \ref{lem:high-deg-embed}, if there is no blue copy of $T$, then there are at most $4d + 2m \le 5d$ vertices whose red out-degree is at most $2d$. Similarly, there are at most $5d$ vertices with red in-degree at most $2d$. Hence, there is a vertex $v$, both of whose red in- and out-degrees are at least $2d$; we assign $u$ to $v$. We now embed, one by one, the trees in the forest $S \setminus \{u\}$. Let $S'$ be such a tree, and suppose that the edge between $u$ and $S'$ in $S$ is directed away from $u$. We seek to embed $S'$ in red within the set of out-neighbours of $v$ not already used. As there are $2d$ such neighbours, and at most $n \le d$ were used, we have $d = a_{l-1}(n+m)$ candidates for the embedding. By induction, either the desired red copy of $S$ exists, or we find a blue copy of $T$.
\end{proof}

We now prove Theorem \ref{thm:complete-trees-two} when both $T_1$ and $T_2$ are out or in-directed.

\begin{lem}\label{lem:outintree-vs-outintree}
    Let $T$ and $S$ be out or in-directed trees. Then $\R(T,S) \le c_2(|T|+|S|)$, for some constant $c_2$.
\end{lem}

\begin{proof}
    Let $n=|T|,m=|S|$, we proceed by induction on $n+m$. The base cases, when $n=1$ or $m=1$ are trivial. Without loss of generality, $n \ge m$ and $T$ is out-directed. Consider a $2$-colouring of $\dircomp{N}$, where $N = c_2(n+m)$.
    
    Let $T'$ be the $16$-core of $T$; then $T'$ has at most $16$ leaves.

    We call a vertex \emph{red} if it has red out-degree at least $c_2\cdot(n/16+m)+n$. Write $c_1 = c_1(16)$ (see \Cref{lem:bounded-leaf-vs-any}). If there are at least $c_1(n+m)$ red vertices, then by \Cref{lem:bounded-leaf-vs-any} we can find a blue copy of $S$ or a red copy of $T'$ among the red vertices. In the former case, we are done, so we assume the latter. In this case, we find copies of the missing subtrees, one by one, in the red out-neighbourhoods of the corresponding vertices of $T'$, avoiding used vertices. By definition of red vertices, and as at most $n$ vertices are used at any given moment, we always have at least $c_2(n/16 + m)$ candidates for the embedding, hence we are always able to find a red copy of the desired subtree (or we find a blue $S$, and are done). We thus find a red copy of $T$ in this case. From now on, we assume that there are more than $c_1(n+m)$ red vertices.
        
    We remove all red vertices and find that in the remaining graph every vertex has blue out-degree at least the following. 
    \begin{align*}
        N - 1 - c_2(n/16+m) - n - c_1 (n+m) + 1
        \ge \,\, & c_2(n+m) - c_2(n/16+m) - (n+m) - c_1 (n+m)  \\
        = \,\, & \frac{15}{16}c_2 n - (n+m) - c_1 (n+m) \\
        \ge \,\, & \frac{15}{32}c_2 (n+m) - (n+m) - c_1 (n+m) \\
        \ge \,\, & \frac{7}{16}c_2 (n+m) = \frac{7}{16}N,
    \end{align*} 
    where the first inequality follows from the assumption that $n \ge m$, and the second inequality follows by taking $c_2 \ge 32(c_1+1)$. 

    
    
   	Assuming that there is no blue $S$, Lemma \ref{lem:dfs} (DFS), applied to the graph of blue edges induced by the non-red vertices, gives a partition $\{U,X,Y\}$ of the remaining vertices with the following properties: $|U| < m$; $|X| = |Y|$; and every vertex in $X$ has at most $m$ blue out-neighbours in $Y$. Hence, the minimum blue out-degree in $X$ is at least $7N/16 - m-|U| > 7N/16 - 2m$, so the maximum red out-degree in $X$ is smaller than $|X|+2m-7N/16$. Lemma \ref{lem:high-deg-embed} (with $k=|X|+2m-7N/16,l=|X|/2-k$) now implies that there is any blue tree on at most $|X|/2-(|X|+2m-7N/16)=7N/16-|X|/2-2m\ge m$ vertices, where we used $|X| \le N/2$ and $N \ge 16m$. So, there is a blue copy of $S$ in $X$.
\end{proof}
This essentially completes the proof of Theorem \ref{thm:complete-trees-two}.
\begin{proof}[ of Theorem \ref{thm:complete-trees-two}]
	Theorem \ref{thm:complete-trees-two} follows by combining \Cref{lem:outintree-vs-outintree} and Lemma \ref{lem:reorientation}.
\end{proof}

\subsection{More colours}\label{sec-more-col}
	Our next aim is to prove Lemma \ref{lem:bounded-deg-tree-vs-indep-set} below, which we shall need in order to prove Theorem \ref{thm:complete-trees-m} for more than two colours. To this end, we focus on out-directed trees. We say that an out-directed tree with root $r$ is \textit{level-regular} if all vertices at distance $i$ from $r$ have the same out-degree, for every $i \ge 1$. Recall that an out-leaf in an out-directed tree is a leaf whose out-degree is $0$. We will need the following preliminary result.
    
    \begin{figure}[h]
        \caption{An example of a level-regular tree.}
        \includegraphics{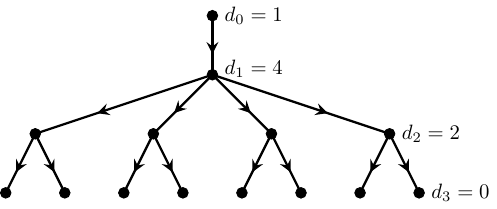}
        \label{fig:sym-tree}
    \end{figure}
    
\begin{lem}\label{lem:sym-tree}
    Let $T$ be an out-directed tree with $l$ out-leaves. There is a level-regular tree $T'$ such that $T$ is a subgraph of $T'$; $T'$ has at most $l^l$ out-leaves; and $|T'| \le l^l |T|$. 
\end{lem}
\begin{proof}        
    Let $d_i$ denote the maximum out-degree among vertices of $T$ at distance $i\ge 0$ from $r$, and let $k$ be the maximum distance of a vertex of $T$ from the root.
    We construct $T'$ as follows. We start by taking $T'$ to be a singleton and $i $ to be $0$, and at step $i \ge 1$ we add exactly $d_{i-1}$ children to each vertex which was added in the previous step. Note that $T'$ contains $T$ as a subtree; indeed, we can greedily embed $T$ in $T'$, starting by embedding the root of $T$ in the root of $T'$.
    By construction, $T'$ has exactly $d_0 \cdots d_{i-1}$ vertices at distance $i$ from the root. In particular, the number of out-leaves of $T'$ is $d_0 \cdots d_{k-1}$ and $|T'| \le d_0 \cdots d_{k-1} |T|$.
    Note that by Lemma \ref{lem:deg-leaf} we have $1+(d_0-1)+\ldots+(d_{k-1}-1)\le l$, so $d_i \le l$ for every $i$, and at most $l$ of the $d_i$'s are greater than $1$. In particular, $|T'| \le d_0 \cdots d_{k-1} |T|\le l^l|T|$ and $T'$ has at most $d_0 \cdots d_{k-1}  \le l^l $ leaves.        
\end{proof}

The following result, interesting in its own right, is the main tool which enables us to prove the result for more colours.
\begin{lem}\label{lem:bounded-deg-tree-vs-indep-set}
Let $T$ be an out-directed tree on $n$ vertices, with $l$ out-leaves. There is a constant $c_l$ such that any directed graph $G$ on $c_lnm$ vertices either contains $T$ as a subgraph or has an independent set of size $m$.
\end{lem}

\begin{proof}
     We first prove the special case of the lemma where $T$ is additionally assumed to be level-regular. We prove this special case by induction on $l$. The base case $l=1$, so $T$ is a directed path, follows from Theorem \ref{thm:ghrv} with $c_1 = 1$. For the step, we assume that there is a constant $c_{l-1}$ which satisfies the setting of the problem, for trees with $l-1$ out-leaves. We assume that there is no independent set in $G$ of size $m$ and will show that there is a copy of $T$ in $G$.
        
    Let $r$ be the root of $T$, and let $u$ be a vertex of out-degree at least $2$, whose distance from $r$, which we denote by $t$, is minimal among vertices of out-degree at least $2$ in $T$. Let $d$ be the out-degree of $u$ (so $d \ge 2$), and note that $T$ consists of a directed path from $r$ to $u$ (of length $t$; note that $t$ could be $0$), and a level-regular out-directed tree rooted at $u$. Let $v_1, \ldots, v_d$ be the children of $u$ in $T$. Denote by $T'$ the subtree of $T$, obtained by removing the subtrees rooted at $v_i$ for $i \ge 2$ (see Figure \ref{fig:bounded-leaves}). Note that $T'$ is a level-regular out-directed tree, which has fewer out-leaves than $T$ (as the out-leaves which are descendants of $v_i$, where $i \ge 2$, are removed, while no new out-leaves are introduced). Hence, by induction on $l$ and by the assumption that there is no independent set of size $m$, there is a copy of $T'$ in every induced subgraph of $G$ of order $c_{l-1} nm$.
    \begin{figure}[h]
        \caption{An illustration of the above notation.}
        \includegraphics{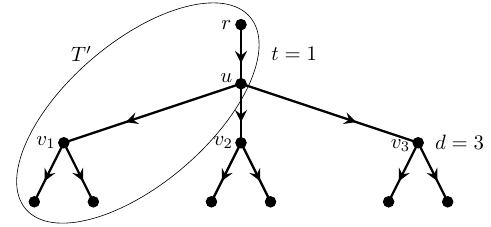}
        \label{fig:bounded-leaves}
	\end{figure}
    
    Let $T_1, \ldots, T_k$ be a maximal collection of vertex-disjoint copies of $T'$ in $G$. Then the vertices of these copies of $T'$ cover all but at most $c_{l-1} nm$ vertices of $G$, hence $c_l nm = |G| \le kn + c_{l-1}nm$. It follows that $k \ge (c_l - c_{l-1})m \ge 2l\cdot m$, where the last inequality follows by letting $c_l \ge c_{l-1} + 2l$. Denote by $u_i$ the vertex in $T_i$ that corresponds to $v_1$ in $T$, and let $U = \{u_1, \ldots, u_k\}$. We claim that there is a vertex in $U$ with out-degree at least $l$ in $U$. Indeed, suppose otherwise. Then, in every subset $W$ of $U$, the number of edges is at most $(l-1)|W|$, hence there is a vertex of in-degree at most $l-1$, so its degree in the underlying graph of $G[W]$ is at most $2l-2$. It follows that the underlying graph of $G[U]$ is $(2l-2)$-degenerate; in particular, it has chromatic number at most $2l - 1$. Therefore, there is an independent set of size at least $|U| / (2l-1) \ge m$ in $U$, a contradiction. Hence, there is a vertex $w_0 \in U$ whose out-degree in $U$ is at least $l$. 
    
    Note that $l \ge d$ by Lemma \ref{lem:deg-leaf}. We may thus pick $d$ out-neighbours of $w_0$ in $U$, denoted by $w_1, \ldots, w_d$. Let $i_j$ be such that $w_j \in T_{i_j}$, where $0 \le j \le d$, recall that $T_{j}$ are all disjoint. We embed $u$ at $w_0$ and combine it with subtrees of $T_{i_j}$ rooted at $w_j$ with the path from the child of the root to $w_0$ in $T_{i_0}$, using the edges from $w_0$ to $w_1, \ldots, w_d$. This forms a copy of $T$ (see Figure \ref{fig:bounded-leaves-2}), as required.
    \begin{figure}[h]
        \caption{An illustration of the embedding described above, for the tree $T$ which appears in Figure \ref{fig:bounded-leaves}.}
        \includegraphics{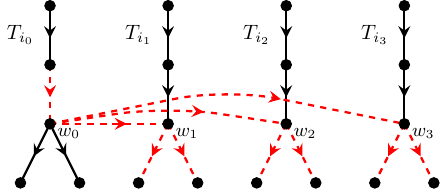}
        \label{fig:bounded-leaves-2}
	\end{figure}
    
    This completes the proof of Lemma \ref{lem:bounded-deg-tree-vs-indep-set} for level-regular trees. The general case, where the tree is not assumed to be level-regular, follows from the level-regular case by Lemma \ref{lem:sym-tree}.
\end{proof}
We are now ready to give the proof of our main result of this section, Theorem \ref{thm:complete-trees-m}.
\begin{proof}[ of Theorem \ref{thm:complete-trees-m}]
	We prove the theorem by induction on triples $(k, l, \sum_{i = 1}^k |T_i|)$, ordered lexicographically, where $l$ denotes the number of trees among $T_1, \ldots, T_k$ that are not in- or out-directed. Specifically, we show that there exist constants $c_{k,l}$ such that 
	\begin{equation} \label{eqn:comp-many-cols}
		\R(T_1, \ldots, T_k) \le c_{k, l} |T_1| \cdots |T_k| \left(\sum_{i = 1}^k \frac{1}{|T_i|}\right).
	\end{equation}
	For the base of the induction, note that the case $k = 2$ follows directly from Theorem \ref{thm:complete-trees-two}.	
	Below, we will show how to prove (\ref{eqn:comp-many-cols}) when for every so-called frequent colour $i$ (defined below), the tree $T_i$ is in- or out-directed. This includes the case where all trees are in- or out-directed, i.e.\ when $l = 0$. 
	Finally, if one of the trees $T_i$ is a single vertex, the assertion clearly follows. 
	
	Let $n = c_{k, l} |T_1| \cdots |T_k| \sum_{i = 1}^k \frac{1}{|T_i|}$ and let $G$ be a $k$-colouring of $\dircomp{n}$; denote its vertex set by $V$, and let $G_i$ be the graph of its edges in colour $i.$ We assume, for the sake of contradiction, that $G_i$ does not contain a copy of $T_i$ for every $i \in [k]$.
	
	We say that a colour $i$ is \emph{$\alpha$-frequent} in a set $U$, if at least $\alpha$ proportion of the edges in $U$ have colour $i$.
	Denote $\alpha = \frac{1}{3\cdot (48k)^2}$.	
	We start by showing that we may assume that for every $\alpha$-frequent colour $i$ in $V$, the tree $T_i$ is in- or out-directed.
	Indeed, suppose that there is an $\alpha$-frequent colour $i$ such that $T_i$ is not in- or out-directed. Without loss of generality, $i = 1$.
	Denote $S_1 = S_1(T_1, \ldots, T_k)$. Then, by Lemma \ref{lem:reorient-one}, 
	\begin{align*}
		\R(T_1, \ldots, T_k) 
		& \le \frac{8}{\alpha} \left(\R(S_1, T_2, \ldots, T_k) + |T_1|\right) \\
		& \le \frac{8}{\alpha} \left(c_{k, l-1} \cdot|S_1| \cdot |T_2| \cdots |T_k| \cdot \left( \frac{1}{|S_1|} + \sum_{i = 2}^k \frac{1}{|T_i|}\right) + |T_1|\right) \\
		& \le \frac{8 (c_{k, l-1} + 1)}{\alpha}  \cdot|T_1| \cdots |T_k| \cdot \sum_{i = 1}^k \frac{1}{|T_i|}\\
		& < c_{k, l} \cdot |T_1| \cdots |T_k| \cdot \sum_{i = 1}^k \frac{1}{|T_i|} = n,
	\end{align*}
	where the last inequality follows by taking $c_{k,l} > \frac{8c_{k, l-1} + 1}{\alpha}$. This is a contradiction; therefore, from now on we may assume that for every frequent colour $i$ in $V$, the tree $T_i$ is in- or out-directed.
	In particular, as explained above, this covers the case where $l = 0$, i.e.\ all the trees $T_i$ are in- or out-directed.
	
	\begin{claim}\label{claim:mindeg-core}
		Let $i$ be a $\frac{1}{3k}$-frequent colour in a set $U$ of size at least $\frac{n}{48k}$. Then there exists a subset $W$ of $U$ of size at least $\frac{1}{48k}|U|$, within which there is no copy of $T_i'$ coloured $i$, where $T_i'$ is some subtree of $T_i$ of order at most $\beta|T_i|,$ for $\beta = \frac{1}{(50k)^2}$.
    \end{claim}
   
    \begin{proof}
    		Note that $i$ is $\alpha$-frequent in $V$, hence, $T_i$ is in- or out-directed.
		We assume, without loss of generality, that $T_i$ is out-directed and $i=1$. 
		By Lemma \ref{lem:split}, there is a $\frac{|U|}{24k}$-mindegree pair $(X, Y)$ in $U$.
		Let $S$ be the $\beta^{-1}$-core of $T_1$. If there is no copy of $S$ in $G_1[X]$, then, by Lemma \ref{lem:bounded-deg-tree-vs-indep-set}, as $S$ has at most $\beta^{-1} = (50k)^2$ out-leaves, there is an independent set in $G_1[X]$ of size at least the following, where $c = c_{(50k)^2}$ is the constant given by Lemma \ref{lem:bounded-deg-tree-vs-indep-set}.
		\begin{align*}
			\frac{|X|}{c |S|}
			\ge \frac{|U|}{24k \cdot c |S|} 
			\ge \frac{n}{48k \cdot 24 k \cdot c |T_1|}
			& \ge \frac{c_{k, l}}{2(24k)^2 \cdot c} \cdot |T_2| \cdots |T_k|\cdot\sum_{i=2}^k \frac{1}{|T_i|} \\
			& \ge c_{k-1, l} \cdot |T_2| \cdots |T_k|\cdot\sum_{i=2}^k \frac{1}{|T_i|}.
		\end{align*} 
		For the last inequality, we assume $c_{k, l} \ge 2(24k)^2 \cdot c \cdot c_{k-1, l}$. By induction on $k$, we find that there is a copy of $T_i$ in colour $i$, for some $i \ge 2$, a contradiction. Hence, there is a copy of $S$ in $G_1[X]$.
		
		Let $S_1, \ldots, S_r$ be the trees comprising of the forest $T_1 \setminus S$, and denote by $u_i$ the vertex in $X$ that corresponds to the vertex in $S$ that sends an edge towards $S_i$, for $i \in [r]$. We attempt to embed the trees $S_i$, one by one, in the out-neighbourhood of $u_i$ in $G_1[Y]$, such that the embedded copies of $S, S_1, \ldots, S_r$ are vertex-disjoint. If we are able to embed all the trees $S_i$, we obtain a copy of $T_1$ in colour $1$, a contradiction. Hence, for some $i \le r$, we fail. Note that the set $W$ of vertices in the out-neighbourhood of $u_i$ in $G_1[Y]$ which were not already used, has size at least $\frac{|U|}{24k} - |T_1| \ge \frac{U}{48k}$. By choice of $i$, there is no copy of $S_i$ in $G_1[W]$, and $S_i$ is a subtree of $T_1$ of order at most $\beta |T_1|$. Hence, $W$ satisfies the requirements of the claim.
    \end{proof}

	Without loss of generality, suppose that $1$ is a majority colour, so, it is $\frac{1}{k}$-frequent in $V$; in particular, it is $\frac{1}{3k}$-frequent. By Claim \ref{claim:mindeg-core}, there is a set $A$ of size at least $\frac{n}{48k}$, and a subtree $T_1'$ of $T_1$ of order at most $\beta |T_1|$, such that there is no copy of $T_1'$ in colour $1$ in $A$. 
	Note that, $1$ is not $2/3$-frequent in $A$. Indeed, this would, by Lemma \ref{lem:bidirect-embed}, imply that we can find any tree of size at most $|A|/12 \ge |T_1'|$ in colour $1$ in $A$, hence there is a copy of $T_1'$ of colour $1$ in $A$, a contradiction.
    
	It follows that there exists a colour $i \neq 1$, such that $i$ is $\frac{1}{3k}$-frequent in $A$; without loss of generality, $i = 2$.
	By Claim \ref{claim:mindeg-core}, there is a subset $B$ of $A$, of size at least $|A|/48k \ge n/(48k)^2$, in which there is no copy of $T_2'$ in colour $2$, where $T_2'$ is a subtree of $T_2$ of order at most $\beta |T_2|$. By induction on $\sum_{i=1}^k |T_i|$, the following holds.
    \begin{align*}
        \frac{n}{(48 k)^2} 
        & \le \,  \R\left(T_1', T_2', T_3, \ldots, T_k \right)  \\
        & \le \,  c_{k,l} \cdot \Big( \,|T_2'| \! \cdot \! |T_3| \! \cdots \! |T_k| \, + \,  |T_1'| \! \cdot \! |T_3| \! \cdots \! |T_k| \, + \, |T_1'| \! \cdot \! |T_2'| \! \cdot \! |T_3| \! \cdots \! |T_k| \! \cdot \! \sum_{i = 3}^k \frac{1}{|T_i|} \, \Big) \\
        & < \, \beta \cdot c_{k,l} \cdot \Big( \, |T_1| \cdots |T_k| \cdot \sum_{i=1}^k \frac{1}{|T_i|} \, \Big) \\
        & =  \frac{n}{(50k)^2}.
    \end{align*} 
    This is a contradiction, completing the proof of Theorem \ref{thm:complete-trees-m}. 
\end{proof}

\subsection{Lower bound}
We conclude this section with the following lower bound, which shows that Theorem \ref{thm:complete-trees-m} is tight for directed paths, up to a constant factor.
\begin{prop}\label{prop:low-paths}
	Let $k \ge 2$, then	$\R(\dirpath{2n+1}, k+1) \ge 4 n^{k}.$
\end{prop}

\begin{proof}
    We first construct a $k$-colouring of the edges of $\dircomp{m}$, for $m=n^{k-1}$, for which there is no monochromatic directed path of order $2n$.
    We represent vertices of $\dircomp{m}$ by $(k-1)$-tuples $(\alpha_1, \ldots, \alpha_{k-1})$ with $\alpha_i \in [n]$.
    We colour the edges as follows, where $a=(\alpha_1, \ldots, \alpha_{k-1})$ and $b = (\beta_1, \ldots, \beta_{k-1})$. It is convenient to define the colouring in steps, for $i$ from $1$ to $k$.
\begin{itemize}
    \item 
    		Step $1$: If $\alpha_1 > \beta_1$ then we colour $ab$ with colour $1$.
    \item 
    		Step $i$, where $1 < i < k$: if $ab$ has not been coloured in any step $j<i$, we colour $ab$ with colour $i$ if $\alpha_i > \beta_i$, or if $\alpha_i = \beta_i$ and $\alpha_{i-1} < \beta_{i-1}$.
    \item 
    		Step $k$: we colour all remaining edges with colour $k$. 
\end{itemize}
	
	Note that if $ab$ is coloured $k$, then $\alpha_{k-1} < \beta_{k-1}$. Indeed, clearly $\alpha_i \le \beta_i$ for every $i \in [k-1]$. Let $i$ be maximal such that $\alpha_i < \beta_i$ (note that such $i$ exists, as otherwise $a = b$). Then if $i \le k-2$, the edge $ab$ would have been coloured by step $i + 1 \le k - 1$, a contradiction.
	
	Let $P$ be a monochromatic directed path; denote the colour of its edges by $i$.
	If $i = 1$, then the first coordinate of the vertices of the path (strictly) decreases along the edges of the path, hence $|P| \le n$.
	If $1 < i < k$, then along the path the $i$-th coordinates decrease and the $(i - 1)$-th coordinates increase, and at least one strictly decreases or increases. Hence $|P| \le 2n-1$.
	Finally, if $i = k$, then the $k-1$-coordinates strictly increase along the path, hence $|P| \le n$.
	
	Let $G_1$ and $G_2$ be two disjoint copies of the above example.
	Colour all edges from $G_1$ to $G_2$ by colour $1$, and colour all edges from $G_2$ to $G_1$ by colour $k$. The resulting graph, which we denote by $G$, is a $k$-colouring of $\dircomp{m}$, where $m = 2n^{k-1}$, without monchromatic paths of order $2n+1$. 
	
	Finally, another factor of $2$ can be gained as follows. Note that each of the colours of $G$ spans an acyclic graph. Let $H$ be a complete directed graph of order $4n^k$ obtained by replacing each vertex $u$ of $G$ by a set $S_u$ of size $2n$. Colour edges inside $S_u$ by colour $k+1$, and for distinct $u, v$ colour the edges from $S_u$ to $S_v$ according to the colour of $uv$ in $G$. $H$ is a $(k+1)$-coloured complete graph of order $4n^k$ which contains no monochromatic path of order $2n+1$ in this colouring, as desired.
\end{proof}
    
\section{Concluding remarks and open problems}\label{sec-conc-rem}

In this paper we considered two very natural analogues of Ramsey theory for directed graphs. Specifically, in both oriented and directed cases we have found bounds on Ramsey numbers for oriented trees, which are tight up to a constant factor. 

Burr and Erd\H{o}s \cite{burr73} initiated the study of Ramsey numbers of bounded degree graphs in 1975. They conjectured that the Ramsey number of bounded degree graphs is linear in their size. This was subsequently proved by Chv\'atal, R\"odl, Szemer\'edi and Trotter \cite{chvatal83}. The dependence of the constant factor on the maximal degree in this bound was later improved, first by Eaton \cite{eaton98}, then by Graham, R\"odl and Ruci\'nski \cite{graham00} and the currently best bound is due to Conlon, Fox, Sudakov \cite{conlon12}. As a natural analogue of this problem, it would be interesting to determine the behaviour of directed or oriented Ramsey numbers of acyclic bounded degree digraphs; note that it is necessary here to require the graph to be acyclic, because both oriented and directed Ramsey numbers of a directed cycle are infinite. It is worth noting that even the case of one colour, in the oriented setting, is of interest; in particular, given $d$, is there a constant $c = c(d)$ such that every tournament of order at least $cn$ contains every acyclic  graph of order $n$ and maximum degree at most $d$? 

While our results are all tight up to a constant factor, it would be interesting to determine the exact values of the directed or oriented Ramsey numbers of certain graphs. For example, in the oriented case the $k$-colour Ramsey number of a path is known: $\RT(\dirpath{n},k)=(n-1)^k+1$. On the other hand, in the directed case $\R(\dirpath{n},k)$ is only known precisely for $k=2$. Already the case $k=3$ is open, Proposition \ref{prop:low-paths}, combined with the standard GHRV based argument, would give $n^2 \le \R(\dirpath{n},3) \le 2n^2$. 

Our results Theorem \ref{thm:tournament-trees} and \ref{thm:complete-trees} are tight up to a constant factor if we are only given the information on the order of the tree. However, in the oriented setting, Theorem \ref{thm:tournament-path} shows that if we know more about the structure of the tree then we are sometimes able to get a significantly better bound. It would be interesting to know whether Theorem \ref{thm:tournament-path} can  be generalised to arbitrary trees, as this would show that the longest directed subpath of the tree represents the main obstruction to the existence of  a monochromatic copy of the tree. It would also be interesting to determine if such a result holds in the directed Ramsey case as well.   

Burr's conjecture, if true, would imply our result for oriented Ramsey number of trees, Theorem \ref{thm:tournament-trees}, although it would not be able to give an improvement, such as the one discussed in the previous paragraph. In the directed case, the main part of our argument is for two colours, where Burr's conjecture would again not be helpful, but it would imply the result for more colours directly. In fact, our argument for more colours uses an intermediate step towards Burr's conjecture, namely Lemma \ref{lem:bounded-deg-tree-vs-indep-set}, which says that if in a directed graph of order $N$, there is no copy of a bounded degree tree of order $n$, then we can find an independent set of size $\Omega(N/n)$. It would be interesting to determine whether this result can be generalised to arbitrary trees.
For arbitrary trees, it is not hard to show that if a graph on $N$ vertices does not contain some oriented tree of order $n$, then there is an independent set of size $\Omega(N/(n \log N))$. Note that for $N$ much larger than $n$, this is weaker than $\Omega(N/n^2)$ implied by Burr's result. As an intermediate step towards a proof of Burr's conjecture, it would be interesting to determine whether the $\Omega(N/n)$ result can be generalised to arbitrary trees.

\subsection*{Acknowledgements}
    
    We would like to thank Sven Heberle and the anonymous referees for their insightful comments.

    \providecommand{\bysame}{\leavevmode\hbox to3em{\hrulefill}\thinspace}
\providecommand{\MR}{\relax\ifhmode\unskip\space\fi MR }
\providecommand{\MRhref}[2]{%
  \href{http://www.ams.org/mathscinet-getitem?mr=#1}{#2}
}
\providecommand{\href}[2]{#2}

\end{document}